\documentclass[a4paper]{article}
\usepackage{amssymb, latexsym, amsmath, amsthm}
\usepackage{a4wide}

\usepackage{color}

\newtheorem*{thm1*}{Theorem A}
\newtheorem*{thm2*}{Theorem B}

\newtheorem{theorem}{Theorem}[section]
\newtheorem{lemma}[theorem]{Lemma}
\newtheorem{proposition}[theorem]{Proposition}
\newtheorem{corollary}[theorem]{Corollary}

\theoremstyle{definition}

\newtheorem{remark}[theorem]{Remark}
\newtheorem{definition}[theorem]{Definition}

\newcommand{\K}{\mathbb{K}}

\newcommand{\Z}{\mathcal{Z}}
\newcommand{\N}{\mathbb{N}}

\begin{document}

\title{Classification of factorial generalized down-up algebras}
\author{St\'ephane Launois\thanks{I am grateful for the full financial support of EPSRC first grant \textit{EP/I018549/1}.} \ and Samuel A. Lopes\thanks{Research funded by the European Regional Development Fund through the programme COMPETE and by the Portuguese Government through the FCT -- Funda‹o para a Cincia e a Tecnologia under the project \textit{PEst-C/MAT/UI0144/2011}.}}
\date{}
\maketitle

\begin{abstract}
We determine when a generalized down-up algebra is a Noetherian unique factorisation domain or a Noetherian unique factorisation ring.
\end{abstract}

\noindent {\it Keywords:} generalized down-up algebra; Noetherian unique factorisation domain; Noetherian unique factorisation ring.
\\$ $
\\{\it 2010 Mathematics Subject Classification:} 16U30; 16S30.

%%%%%%%%%%%%%%%%%%%%%%%%%%%%%%%%%%%%%%%%%%%%%%%%%%%%%%%%%%%%%%%%%%
%%%%%%%%      Introduction                             %%%%%%%%%%%%%%
%%%%%%%%%%%%%%%%%%%%%%%%%%%%%%%%%%%%%%%%%%%%%%%%%%%%%%%%%%%%%%%%%%
\section*{Introduction}\label{S:int}

Down-up algebras were introduced by Benkart and Roby in~\cite{BR98}, motivated by the study of certain ``down'' and ``up'' operators on posets. In this seminal paper, the highest weight theory for a down-up algebra was developed and a parallel was drawn between down-up algebras and enveloping algebras of Lie algebras, based on the apparent similarity between their respective representation theories and  structural properties. Later, in~\cite{CS04}, Cassidy and Shelton introduced a larger class of algebras which, when defined over an algebraically closed field, contains all down-up algebras. 

Let $\K$ be an algebraically closed  field of characteristic $0$. Fix $r, s, \gamma\in\K$ and $f\in\K[x]$. The {\it generalized down-up algebra} $L=L(f, r, s, \gamma)$ is the unital associative $\K$-algebra generated by $d$, $u$ and $h$, subject to the relations:
$$
[d, h]_{r}+\gamma d=0, \quad [h, u]_{r}+\gamma u=0 \quad \mbox{and} \quad [d, u]_{s}+f(h)=0,
$$
where $[a, b]_{\lambda}:=ab-\lambda ba$. A down-up algebra can be seen as a generalized down-up algebra, as above, with $\deg (f)=1$. 

Noteworthy examples of generalized down-up algebras are the enveloping algebra  of the semisimple Lie algebra $\mathfrak{sl}_{2}$, of traceless matrices of size $2$, which is isomorphic to $L(x, 1, 1, 1)$, and the enveloping algebra  of the $3$-dimensional Heisenberg Lie algebra $\mathfrak{h}$, which is isomorphic to $L(x, 1, 1, 0)$. Another example is the quantum Heisenberg Lie algebra $U_{q}(\mathfrak{sl}^{+}_{3})$, where $q\in\K^{*}$, which can be seen as $L(x, q, q^{-1}, 0)$. Under a mild restriction on the parameters, the algebra of regular functions on quantum affine $3$-space, $\mathcal{O}_{Q}(\K^{*})$, is a generalized down-up algebra of the form $L(0, r, s, 0)$, with $rs\neq 0$. In~\cite{spS90}, Smith defined a class of algebras \textit{similar to the enveloping algebra of $\mathfrak{sl}_{2}$}. Subsequently, Rueda considered in~\cite{sR02} a larger family of algebras, including Smith's algebras. The algebras in Rueda's family are generalized down-up algebras of the form $L(f, 1, s, 1)$, and by setting $s=1$ we retrieve Smith's algebras. Other examples of generalized down-up algebras can be found in~\cite[Secs.\ 5 and 6]{BJ01}.

%The algebras studied by Le Bruyn in~\cite{lLB95} are also generalized down-up algebras with $\deg (f)\leq 2$.

Like down-up algebras, generalized down-up algebras display several features of the structure and representation theory of a semisimple Lie algebra, but their defining parameters allow enough freedom to obtain a variety of different behaviours. An example of this is the global dimension of a generalized down-up algebra, which can be $1$, $2$ or $3$, by \cite[Thm.\ 3.1]{CS04} (for a down-up algebra, the global dimension is always $3$). Similarly, in some cases the centre is reduced to the scalars, but in others it can be large, and there are cases in which the generalized down-up algebra is finite over its centre. Other examples of properties which hold in some generalized down-up algebras and do not in others are: being Noetherian, being primitive, having  all finite-dimensional modules semisimple, and having a Hopf algebra structure. 

Generalized down-up algebras have been studied mostly from the point of view of representation theory (see~\cite{BR98}, \cite{CM00}, \cite{rK01}, \cite{jH04}, \cite{CS04} and \cite{iP07}); their primitive ideals have been determined in~\cite{dJ00}, \cite{iP04}, \cite{PM05}, \cite{iP09} and \cite{iPpre}. In this paper we study generalized down-up algebras from the point of view of noncommutative algebraic geometry, namely, we provide a complete classification of those generalized down-up algebras which are (noncommutative) Noetherian unique factorisation rings (resp.\ domains), as defined by Chatters and Jordan in~\cite{C84} and~\cite{CJ86}.

%As a corollary, we obtain a generalisation of the main result of~\cite{rK08}, by establishing that generalized down-up algebras are maximal orders when certain defining parameters are roots of unity.

%Concerning prime ideals, in~\cite[Sec.\ 6]{dJ00} Jordan applies results from~\cite{dJ93} to determine the height one prime ideals of certain down-up algebras

%The representation theory of (generalized) down-up algebras has been studied extensively (see~\cite{BR98}, \cite{CM00}, \cite{rK01}, \cite{jH04}, \cite{CS04} and \cite{iP07}), 
%and several structural features have been investigated, such as their centres (e.g.~\cite{kZ99} and \cite{jH02}), their homological properties (e.g.~\cite{CS04}), their relation with Hopf algebras (e.g.~\cite{BW01} and \cite{KM03}), their isomorphisms and their automorphisms (\cite{CM00}, \cite{CL09}). Regarding the ideal structure of these algebras, 

An element $p$ of a Noetherian domain $R$ is {\em normal} if $pR=Rp$. In our case, a Noetherian
domain $R$ is said to be a {\em unique factorisation ring}, Noetherian UFR for
short, if $R$ has at least one height one prime ideal, and every height one
prime ideal is generated by a normal element. If, in addition, every height one prime factor of $R$ is a domain, then $R$ is called a {\em unique factorisation domain}, Noetherian UFD for short. As well as the usual commutative
Noetherian UFDs, examples of Noetherian UFDs include certain group algebras of polycyclic-by-finite groups~\cite{kB85} 
and various quantum algebras~\cite{LLR06,LL07} such as quantised coordinate rings of semisimple groups. Unfortunately, the notion of a Noetherian UFD is not closed  
under polynomial extensions. To the opposite, the notion of a Noetherian UFR is closed under polynomial extensions. Moreover, Chatters and Jordan proved 
general results for skew polynomial extensions of the type $R[x;\sigma]$ and $R[x;\delta]$. The general case of skew
polynomial extensions of type $R[x;\sigma, \delta]$ is more intricate and only partial results have been obtained  for a class of ``quantum'' algebras called {\em CGL extensions}~\cite{LLR06}, which includes  (generic) quantum matrices, positive parts of quantum enveloping algebras of semisimple Lie algebras, etc.

Going back to enveloping algebras, it follows from results of Conze in~\cite{nC74} that, over the complex numbers, the universal enveloping algebra of a finite-dimensional semisimple Lie algebra is a Noetherian UFD, and an analogous result holds for a finite-dimensional solvable Lie algebra, by~\cite{C84}. It is thus natural to investigate the factorial properties of generalized down-up algebras. Moreover, by considering cases in which the parameters $r$ and $s$ are roots of unity, we hope to get some insight into the behaviour of enveloping algebras over fields of finite characteristic (see~\cite{aB07} and references therein). Indeed, our analysis yields the following result, which shows that, for generalized down-up algebras, the distinction between a Noetherian UFR and a Noetherian UFD depends only on the existence of torsion in the multiplicative subgroup of $\K^{*}$ generated by $r$ and $s$.

\begin{thm1*}\label{mainUFD}
Let $L=L(f, r, s, \gamma)$ be a generalized down-up algebra with $rs \neq 0$. Then $L$ is a Noetherian UFD if and only if $L$ is a Noetherian UFR and $\langle r, s \rangle$ is torsionfree.
\end{thm1*}

Noetherian generalized down-up algebras can be viewed as iterated skew polynomial rings as well as generalized Weyl algebras (see~\cite{KMP99} and~\cite{CS04}). They also can be described as ambiskew polynomial rings (see~\cite{dJ00}). In his paper~\cite{dJ93}, Jordan determined the height one prime ideals of ambiskew polynomial rings under two additional conditions:
\begin{itemize}
\item \textit{conformality}; recall that $f$ is \textit{conformal} in $L(f, r, s, \gamma)$ if there exists $g\in\K[h]$ such that $f(h)=sg(h)-g(rh-\gamma)$;
\item \textit{$\sigma$-simplicity} (see below for the definition of $\sigma$-simplicity).
\end{itemize}
He then applied these results in~\cite[Sec.\ 6]{dJ00} to determine the height one prime ideals of down-up algebras, under certain technical restrictions arising from~\cite{dJ93}.  Here we consider any Noetherian generalized down-up algebra and obtain the following classification:

\begin{thm2*}\label{mainUFR}
Let $L=L(f, r, s, \gamma)$ be a generalized down-up algebra with $rs \neq 0$. Then $L$ is a Noetherian UFR  except if $f\neq 0$ and one of the following conditions is satisfied:
\begin{enumerate}
\item $f$ is not conformal, $r$ is not a root of unity, and there exists $\zeta\neq \gamma/(r-1)$ such that $f(\zeta)=0$;
\item $f$ is conformal, $\langle r, s\rangle$ is a free abelian group of rank $2$, and there exists $\zeta\neq \gamma/(r-1)$ such that $f(\zeta)=0$;
\item $\gamma\neq 0$, $r=1$, $s$ is not a root of unity, and $f\notin \K$.
\end{enumerate}
\end{thm2*}

%apply the results of~\cite{dJ93} only in the cases covered therein, i.e., Sections~\ref{S:rnru} and~\ref{SS:r1:gn0}. 

\begin{flushleft}
\textbf{Acknowledgments.} The authors would like to thank support from the  \emph{Treaty of Windsor Programme}. They also wish to thank Christian Lomp and Paula Carvalho for helpful discussions concerning the topics of this paper, and  David Jordan for the reference~\cite{nC74}.
\end{flushleft}

%%%%%%%%%%%%%%%%%%%%%%%%%%%%%%%%%%%%%%%%%%%%%%%%%%%%%%%%%%%%%%%%%%
%%%%%%%%     Section 1                                      %%%%%%%%%%%%%%
%%%%%%%%%%%%%%%%%%%%%%%%%%%%%%%%%%%%%%%%%%%%%%%%%%%%%%%%%%%%%%%%%%

\section{Generalized down-up algebras and Factoriality}\label{S:gduaF}

Throughout this paper, $\N$ is the set of nonnegative integers, $\K$ denotes an algebraically closed field of characteristic $0$ and $\K^*$ is the multiplicative group of units of $\K$. If $X$ is a subset of the ring $L$ then the two-sided ideal of $L$ generated by $X$ is denoted by  $\langle X\rangle$; we also write $\langle x_{1}, \ldots, x_{n}\rangle $ in place of 
$\langle\{ x_{1}, \ldots, x_{n}\}\rangle $. Moreover, we denote by $\Z(L)$ the centre of $L$.

Given a polynomial $f=a_{0}+a_{1}x+\cdots +a_{n}x^n\in\K[x]$, with all $a_{i}\in\K$, we define the support of $f$ to be the set $\mathrm{supp}\, (f)=\{ i\mid a_{i}\neq 0 \}$ and the degree of $f$, denoted $\deg (f)$, as the supremum of $\mathrm{supp}\, (f)$. In particular, the zero polynomial has degree $-\infty$. In the context of this paper, a monomial in the variable $x$ is a (nonzero) polynomial of the form $\lambda x^{k}$, for some $\lambda\in\K^{*}$ and some $k\geq 0$.

%%%%%%%%%%%%%%%%%%%%%%%%%%%%%%%%%%%%%%%%%%%%%%%%%%%%%%%%%%%%%%%%%%
%%%%%%%%     Section 1.1                              %%%%%%%%%%%%%%
%%%%%%%%%%%%%%%%%%%%%%%%%%%%%%%%%%%%%%%%%%%%%%%%%%%%%%%%%%%%%%%%%%

\subsection{Noetherian generalized down-up algebras}\label{SS:Ngdua}

Let $f\in\K[x]$ be a polynomial and fix scalars $r, s, \gamma\in\K$. The {\it generalized down-up algebra} $L=L(f, r, s, \gamma)$ was defined in~\cite{CS04} as the unital associative $\K$-algebra generated by $d$, $u$ and $h$, subject to the relations:
\begin{align}\label{E:defGdua1}
dh-rhd+\gamma d&=0,\\\label{E:defGdua2}
hu-ruh+\gamma u&=0,\\\label{E:defGdua3}
du-sud+f(h)&=0.
\end{align}

When $f$ has degree one,  we retrieve all down-up algebras 
$A(\alpha, \beta, \gamma)$, $\alpha$, $\beta$, $\gamma\in\K$, for suitable choices of the parameters of $L$. 

It is well known that $L$ is Noetherian $\iff$ $L$ is a domain $\iff$ $rs\neq 0$. Thus, from now on, we will always assume $rs\neq 0$. 
Moreover we can view $L$ as an iterated skew polynomial ring, 
\begin{equation}\label{E:LitOre}
L=\K[h][d;\sigma][u; \sigma^{-1}, \delta],
\end{equation}
where $\sigma (h)=rh-\gamma$, $\sigma (d)=sd$, $\delta (h)=0$, $\delta (d)=s^{-1}f(h)$. (See \cite{CS04} for more details.) \\

To finish this section, we describe the $\mathbb{Z}$-graduation of $L$ obtained by assigning to the generators the following degrees (see~\cite[Sec. 4]{CS04}):
\begin{equation}\label{E:zgrad}
\deg(u)=1,~\deg(d)=-1,~ \deg(h)=0.
\end{equation}
The decomposition $L=\oplus_{i \in \mathbb{Z}} L_i$ of $L$ into homogeneous components has been described in~\cite[Prop. 4.1]{CS04}: 
$$L_0 = \mathbb{K}[h,ud] \mbox{ is the commutative polynomial algebra generated by $h$ and $ud$},$$
and 
\begin{equation}\label{E:zgrad:homcomp}
L_{-i} =L_0d^i =d^iL_0, \quad L_i =L_0u^i =u^iL_0, \quad \mbox{ for } i>0.
\end{equation}

%%%%%%%%%%%%%%%%%%%%%%%%%%%%%%%%%%%%%%%%%%%%%%%%%%%%%%%%%%%%%%%%%%
%%%%%%%%     Section 1.2                              %%%%%%%%%%%%%%
%%%%%%%%%%%%%%%%%%%%%%%%%%%%%%%%%%%%%%%%%%%%%%%%%%%%%%%%%

\subsection{Conformality and isomorphisms}\label{SS:conf}

When we consider two generalized down-up algebras, say $L=L(f, r, s, \gamma)$ and $\tilde{L}=L(\tilde{f}, \tilde{r}, \tilde{s}, \tilde{\gamma})$, we may denote their canonical generators by $d$, $u$, $h$ and $\tilde{d}$, $\tilde{u}$, $\tilde{h}$, respectively, if any confusion could arise regarding which algebra we are referring to. 

\begin{lemma}
The sets $\left\{ d^{i}\right\}_{i\geq 0}$ and $\left\{ u^{i}\right\}_{i\geq 0}$ are right and left denominator sets in $L$.
\end{lemma}

\begin{proof}
See \cite[1.5]{dJ93}. It follows from \cite[Lem. 1.4]{kG92} that  $\left\{ d^{i}\right\}_{i\geq 0}$ is a right and left denominator set in $L$. Using the anti-automorphism that fixes $h$ and interchanges $d$ and $u$ we obtain the corresponding statement for $\left\{ u^{i}\right\}_{i\geq 0}$.
\end{proof}

Fix the parameters $r, s\in\K^{*}$, $\gamma\in\K$, and consider the linear transformation $s\cdot \mathrm{1}-\sigma$ of $\K[h]$. We denote the image of $p\in\K[h]$ under this transformation by $p^{*}$. Specifically, $p^{*}(h)=sp(h)-p(rh-\gamma)$.

\begin{lemma}\label{L:conf:locd}
Let $L=L(f, r, s, \gamma)$, $p\in\K[h]$ and $\tilde{L}=L(f-p^{*}, r, s, \gamma)$. Consider the denominator sets $D=\left\{ d^{i}\right\}_{i\geq 0}$ in $L$, $\tilde{D}=\{ \tilde{d}^{i}\}_{i\geq 0}$ in $\tilde{L}$ and the corresponding localisations $LD^{-1}$ and $\tilde{L}\tilde{D}^{-1}$. 

There is an isomorphism $\phi: LD^{-1}\rightarrow \tilde{L}\tilde{D}^{-1}$, determined by $\phi (d)=\tilde{d}$, $\phi (h)=\tilde{h}$, $\phi (u)=\tilde{u}+p(\tilde{h})\tilde{d}^{-1}$.
\end{lemma}

\begin{proof} To show the existence of  an algebra endomorphism $\phi: L\rightarrow \tilde{L}\tilde{D}^{-1}$ as stated, the following relations need to be checked in $L(f-p^{*}, r, s, \gamma)\tilde{D}^{-1}$:
\begin{align}
\tilde{d}\tilde{h}-r\tilde{h}\tilde{d}+\gamma \tilde{d}&=0;\\
\tilde{h}\left(\tilde{u}+p(\tilde{h})\tilde{d}^{-1}\right)-r\left(\tilde{u}+p(\tilde{h})\tilde{d}^{-1}\right)\tilde{h}+\gamma \left(\tilde{u}+p(\tilde{h})\tilde{d}^{-1}\right)&=0;\\
\tilde{d}\left(\tilde{u}+p(\tilde{h})\tilde{d}^{-1}\right)-s\left(\tilde{u}+p(\tilde{h})\tilde{d}^{-1}\right)\tilde{d}+f(\tilde{h})&=0.\label{E:rel3phi}
\end{align}
\noindent
As the first two of these relations are immediately checked, we show only~(\ref{E:rel3phi}):
\begin{align*}
\tilde{d}\left(\tilde{u}+p(\tilde{h})\tilde{d}^{-1}\right)&=\tilde{d}\tilde{u}+\tilde{d}p(\tilde{h})\tilde{d}^{-1}\\
&=\tilde{d}\tilde{u}+p(r\tilde{h}-\gamma)\\
&=s\tilde{u}\tilde{d}-(f-p^{*})(\tilde{h})+\left(sp(\tilde{h})-p^{*}(\tilde{h})\right)\\
&=s\left(\tilde{u}+p(\tilde{h})\tilde{d}^{-1}\right)\tilde{d}-f(\tilde{h}).
\end{align*}

As $\phi (d)$ is a unit in $\tilde{L}\tilde{D}^{-1}$, the map $\phi$ above extends (uniquely) to a map $\phi: LD^{-1}\rightarrow \tilde{L}\tilde{D}^{-1}$. Now, similar considerations show the existence of an inverse map $\psi: \tilde{L}\tilde{D}^{-1}\rightarrow LD^{-1}$, such that $\psi (\tilde{d})=d$, $\psi (\tilde{h})=h$, $\psi (\tilde{u})=u-p(h)d^{-1}$. Hence, $\phi$ is bijective.

\end{proof}

Given $r, s, \gamma\in\K$, we say that $f\in\K[h]$ is {\it conformal}  if there is $g$ such that $f=g^{*}$. We also say, somewhat abusively, that $f$ is conformal in $L(f, r, s, \gamma)$. Thus, if $f$ is conformal, then $LD^{-1}$ is isomorphic to $L(0, r, s, \gamma)\tilde{D}^{-1}$. In this case, in particular, the nonzero element $z:=ud-g(h)$ is normal and satisfies the relations $zh=hz$, $dz=szd$ and $zu=suz$.

The following results from~\cite{CL09}  determine when a polynomial $f$ is conformal in $L(f, r, s, \gamma)$.

\begin{lemma}[{\cite[Lem.\ 1.6]{CL09}}]\label{L:conf:gamma0}
Let $f=\sum a_{i}h^{i}$. Then $f$  is conformal in $L(f, r, s, 0)$ if and only if $s\neq r^{i}$ for all $i\in \mathrm{supp}\, (f)$. In that case,  a polynomial $g$ satisfying  $f(h)=sg(h)-g(rh)$  exists and is unique if we impose the additional condition that  $\mathrm{supp}\, (f)=\mathrm{supp}\, (g)$; in particular, $g$ can be chosen so that $\deg(g)=\deg(f)$.
\end{lemma}

\begin{proposition}[{\cite[Prop.\ 1.7]{CL09}}]\label{P:conf:gamma0rnot1}
If $r\neq 1$ then $L(f, r, s, \gamma)\simeq L(\tilde{f}, r, s, 0)$ for some polynomial $\tilde{f}$ of the same degree as $f$. Furthermore, $f$ is conformal in $L(f, r, s, \gamma)$ if and only if $\tilde{f}$  is conformal in $L(\tilde{f}, r, s, 0)$.
\end{proposition}

\begin{proposition}[{\cite[Prop.\ 1.8]{CL09}}]\label{P:conf:ris1}
$f$ is conformal in $L(f, 1, s, \gamma)$  except if $s=1$, $\gamma=0$ and $f\neq 0$.
\end{proposition}

%%%%%%%%%%%%%%%%%%%%%%%%%%%%%%%%%%%%%%%%%%%%%%%%%%%%%%%%%%%%%%%%%%
%%%%%%%%     Section 1.3                              %%%%%%%%%%%%%%
%%%%%%%%%%%%%%%%%%%%%%%%%%%%%%%%%%%%%%%%%%%%%%%%%%%%%%%%%%%%%%%%%%

\subsection{Noetherian unique factorisation rings and domains}\label{SS:ufrd}

In this section, we recall the notions of Noetherian unique factorisation rings and Noetherian unique factorisation domains introduced by Chatters and Jordan (see~\cite{C84,CJ86}). \\

An ideal $I$ in a ring $L$ is called {\it principal} if there exists a normal element $x$ in 
$L$ such that $I=\langle x \rangle  = xL = Lx$.

\begin{definition} \label{def-UFR}
A ring $L$ is called a \textit{Noetherian unique factorisation ring} (Noetherian UFR for short) 
if the following two conditions are satisfied: 
\begin{enumerate}
\item $L$ is a prime Noetherian ring;
\item Any nonzero prime ideal in $L$ contains a nonzero principal prime ideal.
\end{enumerate}
\end{definition}

\begin{definition} \label{def-UFD}
A Noetherian UFR $L$ is said to be a \emph{Noetherian unique factorisation domain} (Noetherian UFD for short) if $L$
is a domain and each height one prime ideal $P$ of $L$ is completely prime;
that is, $L/P$ is a domain for each  height one prime ideal $P$ of $L$. 
\end{definition}

Note that  the generalized down-up algebra $L=L(f,r,s,\gamma)$, with $rs \neq 0$, is Noetherian and has 
finite Gelfand-Kirillov dimension; so, it satisfies the descending chain
condition for prime ideals, see for example, \cite[Cor. 3.16]{KL00}. As, moreover,   $L$ is a prime Noetherian ring, 
we deduce from \cite{CJ86} the following result. 

\begin{proposition}
Let $L=L(f,r,s,\gamma)$ be a generalized down-up algebra  with $rs \neq 0$. Then $L$ is a Noetherian UFR if and only if 
all of its height one prime ideals are principal. 
\end{proposition}

To end this section, we recall a  noncommutative analogue of Nagata's Lemma (in the
commutative case, see \cite[19.20 p.\! 487]{E95}) that allows one to prove that a ring is a Noetherian UFR or a Noetherian UFD by 
proving this property for certain localisations of the ring under consideration.

If $L$ is a prime Noetherian ring and $x$ is a nonzero normal element 
of $L$, we denote by $L_x$ the (right) localisation of $L$ with respect to the 
powers of $x$. 

\begin{lemma}[{\cite[Lem. 1.4]{LLR06}}]\label{nagata}
Let $L$ be a prime Noetherian ring and $x$ a nonzero, nonunit, normal element 
of $L$ such that $\langle x \rangle$ is a completely prime ideal of $L$.
\begin{enumerate}
\item If $P$ is a prime ideal of $L$ not containing $x$ and such that the prime 
ideal $PL_{x}$ of $L_{x}$ is principal, then $P$ is principal.
\item If $L_{x}$ is a Noetherian UFR, then so is $L$. 
\item If  $L_{x}$ is a Noetherian UFD, then so is $L$.
\end{enumerate}
\end{lemma}

%%%%%%%%%%%%%%%%%%%%%%%%%%%%%%%%%%%%%%%%%%%%%%%%%%%%%%%%%%%%%%%%%%
%%%%%%%%     Section 1.4                              %%%%%%%%%%%%%%
%%%%%%%%%%%%%%%%%%%%%%%%%%%%%%%%%%%%%%%%%%%%%%%%%%%%%%%%%%%%%%%%%%

\subsection{Some prime ideals of $L$}\label{SS:ql}

In \cite[2.10]{dJ93}, Jordan defines prime ideals $Q(P)$ which depend on certain  prime ideals $P$ of a subalgebra which, in our setting, is $\K[h]$. It is easy to generalize that construction so as to include the case when $f$ is not conformal in $L$, which we will do below. We give the details only for the prime ideals of $\K[h]$ of the form $\langle h-\lambda\rangle$, with $\lambda\in\K$. The only case remaining concerns $\langle 0\rangle$, the zero ideal of $\K[h]$, which will not be necessary for our discussion and carries additional technical issues.

\begin{lemma} \label{L:ql:pk}
Let $L=L(f, r, s, 0)$ and write $f=\sum a_{i}h^{i}$. Then, for every $k\geq 0$, 
\begin{equation}\label{E:L:SS:ql}
 du^{k}=s^{k}u^{k}d-P_{k}(h)u^{k-1},
\end{equation}
where:
\begin{enumerate}
\item $P_{k}(h)= \sum_{i=0}^{k-1}s^{i}f(r^{-i}h)$;
\item If $f=g^{*}$ then $P_{k}(h)=s^{k}g(r^{1-k}h)-g(rh)$;
\item The coefficient of $h^{m}$ in $P_{k}(h)$ is $a_{m}k$, if $s=r^{m}$, and $a_{m}\frac{(sr^{-m})^{k}-1}{sr^{-m}-1}$  if $s\neq r^{m}$.
\end{enumerate}
In particular, if $f$ is not conformal then $P_{k}\neq 0$ for all $k>0$.
\end{lemma}

\begin{proof}
Equation~(\ref{E:L:SS:ql}) along with parts (a) and (b) follow readily by induction on $k\geq 0$. Part (c) follows from (a). Finally, if $f$ is not conformal, recall from~\cite[Lem. 1.6]{CL09} that there is $m$ such that $a_{m}\neq 0$ and $s=r^{m}$. Thus, the coefficient of $h^{m}$ in $P_{k}(h)$ is nonzero, for $k>0$.
\end{proof}

Let $L=L(f, r, s, 0)$. Fix $\lambda \in \mathbb{K}$ and define the $L$-module $V_{\lambda}$ as follows. As a $\K$-vector space, 
$$
V_{\lambda}=\bigoplus_{i\geq 0}\K v_{i}
$$
and the $L$-action is given by:
\begin{align*}
 h.v_{k}&=r^{k}\lambda v_{k}\\
 u.v_{k}&=v_{k+1}\\
 d.v_{k}&=-P_{k}(r^{k-1}\lambda)v_{k-1},\ \text{for $k\geq 1$},\ \text{and}\quad d.v_{0}=0.
\end{align*}

Assume there is $k>0$ such that $P_{k}(r^{k-1}\lambda)=0$. Then $\bigoplus_{i\geq k}\K v_{i}$ is a proper submodule of $V_{\lambda}$. Let $k>0$ be minimal with this property, and define $M_{\lambda}=\bigoplus_{i\geq k}\K v_{i}$. Thus, $L_{\lambda}:=V_{\lambda}/M_{\lambda}$ is a finite-dimensional representation of $L$. Let $Q_{\lambda}:=\mathrm{ann}_{L}L_{\lambda}$. By the minimality of $k$ it is straightforward to see that $L_{\lambda}$ is simple. Thus, $Q_{\lambda}$ is a primitive ideal; in particular, it is prime.

\begin{remark}~
\begin{enumerate}
\item $M_{\lambda}$, $L_{\lambda}$ and $Q_{\lambda}$ are  defined only if there exists $k>0$ such that $P_{k}(r^{k-1}\lambda)=0$.
\item When $f$ is conformal, this construction is a special case of the construction in \cite[2.10]{dJ93}, where $P$ is the ideal of $\K[h]$ generated by $h-\lambda$ and $Q_{\lambda}=Q(\langle h-\lambda \rangle)$.
\end{enumerate}
\end{remark}

\begin{theorem}\label{T:ql}
 Let $L=L(f, r, s, 0)$. Suppose $\lambda\in\K$ is such that $P_{k}(r^{k-1}\lambda)=0$ for some $k>0$. Then $Q_{\lambda}$ is a non-principal maximal ideal of $L$ containing $d^{k}$ and $u^{k}$.
 
 Furthermore, if $P_{k}\neq 0$ for all $k>0$ (e.g., if $f$ is not conformal) and $Q$ is any prime ideal of $L$ containing $d^{k}$ and $u^{k}$ for some $k>0$, then there exists $\lambda\in\K$ such that $Q_{\lambda}$ is defined and $Q=Q_{\lambda}$.
\end{theorem}

\begin{proof}
This follows essentially as in~\cite[Thm. 2.12]{dJ93}. We give details for completeness.  

Let $\rho: L\rightarrow \mathrm{End}_{\K}(L_{\lambda})$ be the map which defines the representation. Since $L_{\lambda}$ is finite-dimensional and simple, and $\K$ is algebraically closed, Schur's Lemma implies that $\mathrm{End}_{L}(L_{\lambda})$, the centraliser algebra of $L_{\lambda}$, is just $\K$. Thus, by the Jacobson Density Theorem, $\rho$ is onto and induces an algebra isomorphism $L/Q_{\lambda}\simeq \mathrm{End}_{\K}(L_{\lambda})$. As $\mathrm{End}_{\K}(L_{\lambda})$ is simple, the ideal $Q_{\lambda}$ is maximal.

If $Q$ is any prime ideal of $L$ containing $d^{k}$ and $u^{k}$ for some $k>0$, then the proof of~\cite[Thm. 2.12]{dJ93} shows that there is $k>0$ and a prime ideal $P$ of $\K[h]$ such that $P_{k}(r^{k-1}h)\in P$. As we are assuming $P_{k}\neq 0$ for all $k>0$, and $\K$ is algebraically closed, it follows that there is $\lambda\in\K$ such that $P_{k}(r^{k-1}\lambda)=0$. Then, as in the proof of~\cite[Thm. 2.12]{dJ93}, we have $Q_{\lambda}\subseteq Q$, and hence, by the maximality of $Q_{\lambda}$, we obtain $Q=Q_{\lambda}$.

Finally,   $Q_{\lambda}$ is not principal because, by the definition of $L_{\lambda}$, we have $d^{k}, u^{k}\in Q_{\lambda}$. This is a general fact concerning any generalized Weyl algebra $D(\phi, a)$ over a commutative domain $D$ such that $0\neq a\in D$ is not a unit. (Recall, e.g.~\cite[Lem. 2.7]{CS04}, that $L$ is a generalized Weyl algebra, where $D$ is the polynomial algebra in the variables $h$ and $a=ud$.) Nevertheless, we give the specific details for $L$. 

Assume $\xi L$ is a principal ideal of $L$ containing $u^{k}$, for some $k>0$. Then, the equation $\xi x=u^{k}$, for $x\in L$, implies that both $\xi$ and $x$ must be homogeneous, with respect to the $\mathbb{Z}$-grading defined in~(\ref{E:zgrad}). Assume $\xi$ has degree $n<0$. Then we can write $\xi=td^{-n}$ and $x=t'u^{k-n}$, for some $t, t'\in\K[h, ud]$. We have:
\begin{equation*}
u^{k}=(td^{-n})(t'u^{k-n})=t\phi^{-n}(t')d^{-n}u^{-n}u^{k}=t\phi^{-n}(t')\left( \prod_{i=1}^{-n}\phi^{i}(ud)\right)u^{k},
\end{equation*}
where $\phi$ is the automorphism of $\K[h, ud]$ defined by $\phi(h)=rh$ and $\phi (ud)=sud-f(h)$. The above equation implies that $ud$ is a unit in $\K[h, ud]$, which is a contradiction. Hence, $\xi$ has degree $n\geq 0$. Similarly, assuming that $d^{k}\in\xi L$, we conclude that $\xi$ has degree $n\leq 0$. It follows that , if $\xi L$ contains both $u^{k}$ and $d^{k}$, then $\xi\in\K[h, ud]$. But then the equation $\xi (tu^{k})=u^{k}$, for $t\in\K[h, ud]$, implies that $\xi$ is a unit and $\xi L=L$. Thus, no proper ideal of $L$ containing $u^{k}$ and $d^{k}$ can be principal.
\end{proof}

We end this section by pointing out  some principal height one prime ideals which will also be of interest later. 

\begin{lemma}\label{L:ql:h1p}
Let $L=L(f, r, s, 0)$. Then the normal element $h$ generates a height one, completely prime ideal of $L$. Furthermore, if $r$ is a primitive root of unity of order $l\geq 1$ then, for any $\lambda\in\K^{*}$, the central element $h^{l}-\lambda$ generates a height one prime ideal of $L$ which is completely prime if and only if $r=1$.
\end{lemma}

\begin{proof}
 First, notice that $h$ is normal, as $\gamma=0$, and generates a completely prime ideal, as the factor algebra $L/\langle h \rangle$ is either a quantum plane or a quantum Weyl algebra, or one of their classical analogues, in case $s=1$. By the Principal Ideal Theorem (see~\cite[4.1.11]{MR01}), $\langle h \rangle$ has height one. 
 
If $r$ is a primitive root of unity of order $l\geq 1$ then $h^{l}$ is central. Consider the presentation $L=\K[h][d;\sigma][u; \sigma^{-1}, \delta]$ of $L$ as an iterated skew polynomial ring, as given in~(\ref{E:LitOre}) above, with $\sigma (h)=rh$ and $\delta (h)=0$.
It is easy to see that $(h^{l}-\lambda)\K[h]$ is a $\sigma$-prime ideal of $\K[h]$ (i.e., it is a prime ideal in the lattice of $\sigma$-stable ideals of $\K[h]$). It follows, e.g. by~\cite[Prop.\ 2.1]{aB85}, that $h^{l}-\lambda$ generates a prime ideal of $\K[h][d; \sigma]$. In particular, this ideal is  $\sigma$-prime and $\delta$-stable, so it follows by~\cite[Prop.\ 2.1]{aB85} that $\langle h^{l}-\lambda \rangle$ is a prime ideal of $L=\K[h][d;\sigma][u; \sigma^{-1}, \delta]$. Again by the Principal Ideal Theorem, this ideal has height one.

If $l\geq 2$, then $h^{l}-\lambda$ factors nontrivially, as $\K$ is algebraically closed, so $\langle h^{l}-\lambda \rangle$ is not completely prime, by simple degree arguments. Otherwise, if $l=1$ then $r=1$ and the factor algebra $L/\langle h-\lambda \rangle$ is again a quantum plane or a quantum Weyl algebra, or one of their classical analogues, so in this case the ideal $\langle h-\lambda \rangle$ is completely prime.
\end{proof}

%%%%%%%%%%%%%%%%%%%%%%%%%%%%%%%%%%%%%%%%%%%%%%%%%%%%%%%%%%%%%%%%%%
%%%%%%%%     Section 2                              %%%%%%%%%%%%%%
%%%%%%%%%%%%%%%%%%%%%%%%%%%%%%%%%%%%%%%%%%%%%%%%%%%%%%%%%%%%%%%%%%

\section{The case $f$ not conformal}\label{S:fnconf}

Assume $f=\sum a_{i}h^{i}$ is not conformal. Then, by Propositions~\ref{P:conf:gamma0rnot1} and~\ref{P:conf:ris1}, we can assume $\gamma=0$. By Lemma~\ref{L:conf:gamma0}, we can write $f=f_{c}+f_{nc}$, where $f_{c}=g^{*}$ is conformal and $f_{nc}$ is such that $s=r^{i}$ for all $i\in\mathrm{supp}\, (f_{nc})$. Such a decomposition $f=f_{c}+f_{nc}$ is unique, and $f_{nc}\neq 0$, as $f$ is not conformal.

\begin{lemma}
 Let $L=L(f, r, s, 0)$ with $f$ not conformal. There is $j\in\mathrm{supp}\, (f)$ such that $s=r^{j}$ and $f_{nc}=h^{j}F$, where $0\neq F\in\K[h]\cap\Z(L)$. Furthermore:
\begin{enumerate}
\item If $r$ is not a root of unity, then $F\in\K^{*}$;
\item If $r$ is a root of unity of order $l\geq 1$, then $F(h)=G(h^{l})$ is a polynomial in the central indeterminate $h^{l}$. 
\end{enumerate}
\end{lemma}

\begin{proof}
 Let us write $f_{nc}= \sum_{i\in\mathrm{supp}\, (f_{nc})} a_{i}h^{i}$. Let $j=\min \mathrm{supp}\, (f_{nc})$. Thus, $s=r^{j}$ and we can write $f_{nc}=h^{j}F(h)$, where $F(h)=\sum_{i\in\mathrm{supp}\, (f_{nc})} a_{i}h^{i-j}$. Given $i\in\mathrm{supp}\, (f_{nc})$, we have $i-j\geq 0$ and $r^{i}=s=r^{j}$, so $r^{i-j}=1$.
 
 If $r$ is not a root of unity, then $i=j$ and $F(h)\in\K^{*}$. Otherwise, if $r$ is a primitive $l$-th root of unity, then $l$ divides $i-j$ and $F(h)$ is a polynomial in $h^{l}$, which is thus central, as $\gamma=0$. 
 \end{proof}

\begin{proposition}
Let $L=L(f, r, s, 0)$, with $f$ not conformal, and consider the localisation $LD^{-1}$, where $D=\left\{ d^{i}\right\}_{i\geq 0}$. Then $\left\{ h^{i}\right\}_{i\geq 0}$ is a right and left denominator set in $LD^{-1}$ and the localisation at this set is isomorphic to $\hat{L}=L(F, r, 1, 0)$ localised at the multiplicative set generated by the corresponding elements $\hat{d}$ and $\hat{h}$ in $\hat{L}$, where $f=f_{c}+f_{nc}$ and $f_{nc}=h^{j}F$, as in the previous lemma.
\end{proposition}

\begin{proof}
The first statement follows from the normality of $h$ in $L$. We have already seen in Lemma~\ref{L:conf:locd} that $LD^{-1}$ is isomorphic to $\tilde{L}\tilde{D}^{-1}$, where $\tilde{L}=L(f_{nc}, r, s, 0)$, under an isomorphism that maps $h$ to $\tilde{h}$, $d$ to $\tilde{d}$ and $u$ to $\tilde{u}+g(\tilde{h})\tilde{d}^{-1}$, where $f_{c}=g^{*}$. So it suffices to show that $\tilde{L}$ localised at the multiplicative set generated by $\tilde{d}$ and $\tilde{h}$ is isomorphic to the corresponding localisation of $\hat{L}=L(F, r, 1, 0)$ at the multiplicative set generated by $\hat{d}$ and $\hat{h}$.

It is easy to see that there is an algebra homomorphism $\Phi: L(f_{nc}, r, s, 0)\longrightarrow L(F, r, 1, 0)$ such that $\Phi (\tilde{d})=\hat{d}$, $\Phi (\tilde{h})=\hat{h}$ and $\Phi (\tilde{u})=\hat{u}\hat{h}^{j}$. This homomorphism clearly extends to an isomorphism when we pass to the localisation under consideration.
\end{proof}

Next, we define a (left and right) denominator set $X$ in $L(f, r, s, 0)$, which depends on $r$:
\begin{enumerate}
\item If $r$ is not a root of unity, then $X$ is the multiplicative set generated by $d$ and $h$.
\item If $r$ is a root of unity of order $l\geq 1$, then $X$ is the multiplicative set generated by $d$, $h$ and the central elements of the form $h^{l}-\lambda$, for $\lambda\in\K^{*}$.
\end{enumerate}

\begin{proposition}\label{P:locX}
Let $L=L(f, r, s, 0)$ and assume $f$ is not conformal. Then the localisation of $L$ at the denominator set $X$ defined above is a simple algebra.
\end{proposition}

\begin{proof}
By the previous result, it is enough to assume $L=L(F, r, 1, 0)$, where $F\neq 0$ is either a scalar (if $r$ is not a root of unity) or a polynomial in the central indeterminate $h^{l}$ (if $r$ has order $l\geq 1$). Furthermore, since $F$ is central and invertible in the localisation under consideration ($\K$ is algebraically closed), we can assume $F=1$, on replacing u by $uF^{-1}$. The result then follows from the description below of the prime ideals of $L(1, r, 1, 0)$.
 \end{proof}
 
 \begin{theorem} \label{T:fnconf:spec}
 Let $L=L(1, r, 1, 0)$.
\begin{enumerate}
\item Assume that $r$ is not a root of unity. Then $\mathrm{Spec}(L)=\{\langle 0\rangle, \langle h\rangle \}$. 
\item Assume that $r$ is a primitive $l$-th root of unity, for $l\geq 1$. Then $$\mathrm{Spec}(L)=\{ \langle 0\rangle, \langle h\rangle \} \cup \{\langle h^l-\lambda\rangle ~|~\lambda \in \mathbb{K}^*\}.$$
\end{enumerate}

\end{theorem}

\begin{proof} This follows from the isomorphism $L(1, r, 1, 0)\simeq \mathbb{A}_{1}(\K)[h;\phi]$, where  $\mathbb{A}_{1}(\K)$ denotes the first Weyl algebra over $\K$, generated by $d$ and $u$, subject to the relation $ud-du=1$, and $\phi$ is the automorphism of $\mathbb{A}_{1}(\K)$ defined by $\phi (d)=r^{-1}d$ and $\phi (u)=ru$. Thus, we identify the algebras $L$ and $\mathbb{A}_{1}(\K)[h;\phi]$.  Below we sketch the proof, which relies on the simplicity of $\mathbb{A}_{1}(\K)$ (recall that $\K$ has characteristic $0$).

Firstly, all ideals listed in the statement are prime, e.g. by Lemma~\ref{L:ql:h1p}. We will show that there are no other prime ideals.
There is an $\mathbb{N}$-grading on $\mathbb{A}_{1}(\K)[h;\phi]$ such that the homogenous component of degree $n\geq 0$ is $\mathbb{A}_{1}(\K)h^{n}=h^{n}\mathbb{A}_{1}(\K)$. This grading is, of course, different from the usual $\mathbb{Z}$-grading of $L$ we consider in the paper, but for the remainder of this proof, this is the grading we will consider. As usual, the degree of an element in $\mathbb{A}_{1}(\K)[h;\phi]$ is the maximum of the degrees of its nonzero homogeneous components, i.e., its degree as a polynomial in $h$.

Assume $P\neq \langle 0\rangle$ is a prime ideal of $\mathbb{A}_{1}(\K)[h;\phi]$. Let $0\neq \xi\in P$ be a (not necessarily homogeneous) element of minimum degree, say $n\geq 0$. Then the set of leading coefficients of nonzero elements of $P$ of degree $n$, adjoined with $0$, is easily seen to be an ideal of $\mathbb{A}_{1}(\K)$. As the latter is simple and $\xi\neq 0$, it follows that this ideal contains $1$. Therefore, we can assume that $\xi$ is monic. By the minimality of the degree of $\xi$ and the fact that its leading coefficient is a unit, we can use right and left division algorithms to conclude that $P$ is principal and generated by $\xi$, both on the right and on the left. In particular, $\xi$ is normal. 

Since $\mathbb{A}_{1}(\K)[h;\phi]$ is $\mathbb{N}$-graded, every homogeneous constituent of $\xi$ is normal, so we will first determine the homogeneous elements of $\mathbb{A}_{1}(\K)[h;\phi]$ which are normal. Assume $ah^{i}$ is normal, where $a\in \mathbb{A}_{1}(\K)$ and $i\geq 0$. Then, as $h^{i}$ is itself normal, it follows that $a$ is normal in $\mathbb{A}_{1}(\K)$. Thus, $a\in\K$. This shows, in particular, that the normal elements of $\mathbb{A}_{1}(\K)[h;\phi]$ are polynomials in $h$ with coefficients in $\K$, but not every such polynomial is normal, except if $r=1$. Indeed, suppose $0\neq \xi=\sum_{i\geq 0} \lambda_{i}h^{i}$ is normal, where $\lambda_{i}\in\K$. Then there is $a$ such that $d\xi=\xi a$. It must then be that $a\in\mathbb{A}_{1}(\K)$, by degree considerations, and $\lambda_{i}d=\lambda_{i}\phi^{i}(a)$, for all $i$. If $\lambda_{i}$ and $\lambda_{j}$ are nonzero, then $r^{i}d=\phi^{-i}(d)=a=\phi^{-j}(d)=r^{j}d$, so $r^{i-j}=1$. This implies that we can write $\xi=h^{k}G$, where $k\geq 0$ and  either $G$ is a (nonzero) scalar, if $r$ is not a root of unity, or $G$ is a polynomial in $h^{l}$ with scalar coefficients and nonzero constant term, if $r$ is a primitive $l$-th root of unity. 

As $h$ and $G$ are normal, and $P=\langle \xi\rangle$ is prime, either $h\in P$ or $G\in P$. If the former occurs, then $P=\langle h \rangle$. Otherwise, $k=0$, $r$ is a primitive $l$-th root of unity, and $P=\langle h^l-\lambda\rangle$, for some $\lambda\in\K^{*}$, as $\K$ is algebraically closed and, up to a scalar, $G$ can be factored into central polynomials of the form $h^l-\mu$, for $\mu\in\K^{*}$. This establishes the claim.
\end{proof}

Similarly, we can define a (left and right) denominator set $Y$ in $L(f, r, s, 0)$, which can be obtained from $X$ by replacing $d$ by $u$. Specifically:
\begin{enumerate}
\item If $r$ is not a root of unity, then $Y$ is the multiplicative set generated by $u$ and $h$.
\item If $r$ is a root of unity of order $l\geq 1$ then $Y$ is the multiplicative set generated by $u$, $h$ and the central elements of the form $h^{l}-\lambda$, for $\lambda\in\K^{*}$.
\end{enumerate}

\begin{proposition}\label{P:locY}
Let $L=L(f, r, s, 0)$ and assume $f$ is not conformal. Then the localisation of $L$ at the denominator set $Y$ defined above is a simple algebra.
\end{proposition}

\begin{proof}
 Consider the isomorphism $L(f, r, s, 0)\longrightarrow L(f, r^{-1}, s^{-1}, 0)$, defined by the correspondence $d\mapsto -s^{-1}u$, $u\mapsto d$, $h\mapsto h$. Notice that $f$ is conformal in $L(f, r, s, 0)$ if and only if $f$ is conformal in $L(f, r^{-1}, s^{-1}, 0)$ (if $f(h)=sg(h)-g(rh)$ then $f(h)=s^{-1}G(h)-G(r^{-1}h)$, for $G(h)=-sg(rh)$). Thus, our claim follows from applying our previous result to $L(f, r^{-1}, s^{-1}, 0)$ and the denominator set $X$, and using this isomorphism.
\end{proof}

We can now determine when $L(f, r, s, 0)$ is a Noetherian UFR or a Noetherian UFD, assuming $f$ is not conformal.

\begin{theorem}\label{T:fnconf:class}
Let $L=L(f, r, s, 0)$ and assume $f$ is not conformal. Then $L$ is a Noetherian UFR, except in the case that $f$ is not a monomial  and $r$ is not a root of unity. Moreover, $L$ is a Noetherian UFD if and only if either $r=1$ or if $r$ is not a root of unity and $f$ is a monomial.
\end{theorem}

\begin{proof}

Let us first identify the possible height one primes.

Let $P$ be a height one prime ideal of $L$. If $P$ does not contain any power of $d$ or if $P$ does not contain any power of $u$ then, by Propositions~\ref{P:locX} and~\ref{P:locY}, $P$ must contain either  $h$ or $h^{l}-\lambda$, for some $\lambda\in\K^{*}$ (the latter can occur only when $r$ is a primitive $l$-th root of unity, for some $l\geq 1$), as these elements are normal. But both $h$ and $h^{l}-\lambda$ generate prime ideals, by Lemma~\ref{L:ql:h1p}, so it follows that either $P=\langle h\rangle$ or $P=\langle h^{l}-\lambda\rangle$.

Otherwise, $P$ must contain both a power of $d$ and a power of $u$. Thus, $P=Q_{\lambda}$, for some $\lambda\in\K$, by Theorem~\ref{T:ql}. In particular, $P_{k}(r^{k-1}\lambda)=0$ for some $k>0$ (with the notation of Lemma \ref{L:ql:pk}). Assume that $\lambda=0$. Then  $h\in Q_{0}=P$ and $P=\langle h\rangle$, which is a contradiction, as $\langle h\rangle$ does not contain any power of $d$. So $P=Q_{\lambda}$ for some $\lambda\in\K^{*}$.

To summarise, the possible height one primes of $L$ are:
$\langle h\rangle$;  $\langle h^{l}-\lambda\rangle$ with $\lambda\in\K^{*}$ and  $Q_{\lambda}$ for some $\lambda\in\K^{*}$ such that $P_{k}(r^{k-1}\lambda)=0$, for some $k>0$.

We now distinguish between three different cases.

Let us first consider the case that $f$ is a monomial. Then $P_{k}\neq 0$ is also a monomial and hence the only possibility for $\lambda $ to satisfy $P_{k}(r^{k-1}\lambda)=0$  is $\lambda=0$, which is a contradiction. So the only possible height one primes of $L$ are $\langle h\rangle$ and  $\langle h^{l}-\lambda\rangle$ with $\lambda\in\K^{*}$. They are all principal so that $L$ is a Noetherian UFR. Furthermore, it follows from Lemma~\ref{L:ql:h1p} that $L$ is a Noetherian UFD except if $r$ is a primitive root of unity of order $l\geq 2$.

If $f$ is not a monomial,  then we consider two cases:

\textit{Case 1:} $r$ is a primitive $l$-th root of unity. In this case, $h^{l}-\lambda^{l}$ annihilates $L_{\lambda}$, so $h^{l}-\lambda^{l}\in Q_{\lambda}$. It follows that $Q_{\lambda}$ is not a height one prime. So, as above, the only height one primes of $L$ are $\langle h\rangle$ and  $\langle h^{l}-\lambda\rangle$ with $\lambda\in\K^{*}$, whence the final statement follows.

\textit{Case 2:} $r$ is not a root of unity. As $f$ is not a monomial, there is $\eta\in\K^{*}$ such that $P_{1}(\eta)=f(\eta)=0$. Assume $Q_{\eta}$ does not have height 1. Then it properly contains a nonzero prime ideal $Q$. This ideal $Q$ cannot be of the form $Q_{\lambda}$, as these ideals are maximal, hence either $Q$ does not contain a power of $d$ or $Q$ does not contain a power of $u$. By the first part of our argument, as $r$ is not a root of unity, $Q$ must contain $h$. In particular, $h$ annihilates the module $L_{\eta}$, which is a contradiction, as $\eta\neq 0$. Thus, $Q_{\eta}$ indeed has height one and is not principal, so $L$ is not a Noetherian UFR in this case.
\end{proof}

%%%%%%%%%%%%%%%%%%%%%%%%%%%%%%%%%%%%%%%%%%%%%%%%%%%%%%%%%%%%%%%%%%
%%%%%%%%     Section 3                              %%%%%%%%%%%%%%
%%%%%%%%%%%%%%%%%%%%%%%%%%%%%%%%%%%%%%%%%%%%%%%%%%%%%%%%%%%%%%%%%%

\section{The case $f=0$}\label{S:f0}

We will consider separately the cases $\gamma=0$ and $\gamma\neq 0$.

%%%%%%%%%%%%%%%%%%%%%%%%%%%%%%%%%%%%%%%%%%%%%%%%%%%%%%%%%%%%%%%%%%
%%%%%%%%     Section 3.1                              %%%%%%%%%%%%%%
%%%%%%%%%%%%%%%%%%%%%%%%%%%%%%%%%%%%%%%%%%%%%%%%%%%%%%%%%%%%%%%%%%
\subsection{The case $f=0$ and $\gamma=0$}\label{SS:f0:g0}

In this case, the defining relations of $L=L(0,r,s,\gamma)$ are:
\begin{equation*}
dh=rhd, \quad hu=ruh, \quad du=sud,
\end{equation*}
and $L$ is the so-called \textit{quantum coordinate ring of affine $3$-space} over $\K$. The normal elements $d$, $u$ and $h$ generate pairwise distinct completely prime ideals so, by~\cite[Prop.\ 1.6]{LLR06}, it is enough to show that the localisation $T$ of $L$ with respect to the Ore set generated by these three elements is a Noetherian UFR. Well, by~\cite[1.3(i) and Cor.\ 1.5]{GL98}, the height one prime ideals of $T$ are generated by a single central element, so $T$ is a Noetherian UFR. Thus, $L$ is a Noetherian UFR. We record this result below.

\begin{proposition}\label{P:f0:g0:UFR}
Assume $f=0$ and $\gamma= 0$. Then $L(0,r,s,0)$ is a Noetherian UFR.
\end{proposition}

We conclude this section by studying for which values of $r$ and $s$ the generalized down-up algebra $L(0,r,s,0)$ is a Noetherian UFD.

\begin{proposition}\label{P:f0:g0:UFD}
 Assume $f=0$ and $\gamma= 0$. Then $L=L(0,r,s,0)$ is a Noetherian UFD if and only if $\langle r,s \rangle$ is torsionfree. 
\end{proposition}

\begin{proof}
As was observed above, $L$  is just the quantum coordinate ring of an affine 3-space. 
So we deduce from~\cite[Thm. 2.1]{GL94} that, if $\langle r,s \rangle$ is torsionfree, then all prime ideals of $L$ are completely prime. Thus, the result is proved in this case. 

Now assume that $\langle r,s \rangle$ is not torsionfree. 
First, if $r$ is a root of unity of order $l \geq 2$, then the result follows from Lemma~\ref{L:ql:h1p}. 
So we are left with the cases where $r$ is either 1 or not a root of unity. Before distinguishing between different cases, let us describe our strategy to prove that $L$ is not a Noetherian UFD in these cases. 

If $L$ were a Noetherian UFD, then so would be the localisation $T$ of $L$ at the Ore set generated by the normal elements $h$, $d$, $u$.  
(Note that this is due to the fact that we are localising at elements that are ``$q$-central'' - see also Proposition~\ref{L:rnru:iffUFR}.) This localised algebra $T$ is a quantum torus. More precisely, it is  the quantum torus generated by the three indeterminates $h$, $d$ and $u$, and their inverses $h^{-1}$, $d^{-1}$ and $u^{-1}$, subject to the relations
$$dh=rhd,~hu=ruh,~du=sud.$$
Now it follows from \cite{GL94} that extension and contraction provide mutually inverse bijections between the prime spectrum of $T$ and the prime spectrum of the centre $\Z(T)$ of $T$, and that $\Z(T)$ is a (commutative) Laurent polynomial algebra over $\mathbb{K}$. Moreover, we can compute this centre explicitly, using~\cite[1.3]{GL94}. 
So, to prove that $L$ is not a Noetherian UFD in the remaining cases, we will construct a height one prime ideal of $T$ which is not completely prime. This is achieved by computing the centre of $T$ in each case.

We distinguish between three cases:

{\it Case 1: $r=1$ and $s$ is a root of unity of order $\beta \geq 2$.} In this case, we get
$$\Z(T) =\mathbb{K}[u^{\pm\beta}, h^{\pm 1}, d^{\pm\beta}].$$
Hence, $u^\beta -1$ generates a height one prime ideal in $T$ which is not completely prime. \\
 
{\it Case 2: $r$ is not a root of unity and $s$ is a root of unity of order $\beta \geq 2$.}
In this case, we get
$$\Z(T) =\mathbb{K}[(ud)^{\pm\beta}]. $$ 
Hence, $(ud)^\beta -1$ generates a height one prime ideal in $T$ which is not completely prime. \\

{\it Case 3: $r$ and $s$ are not  roots of unity.} Hence, there exists  $(\alpha_0,\beta_0)$ with $\beta_0>0$ minimal such that $r^{\alpha_0}s^{\beta_0}=1$. In this case, we deduce from~\cite[1.3]{GL94} that 
$$\Z(T)=\mathbb{K}[(u^{\beta_0}h^{\alpha_0}d^{\beta_0})^{\pm 1}].$$

Now observe that the fact that  $\langle r,s \rangle$ is not torsionfree imposes that  $\gcd(\alpha_0,\beta_0)>1$. Hence, $u^{\beta_0}h^{\alpha_0}d^{\beta_0} -1$ generates a height one prime ideal in $T$ which is not completely prime. 
\end{proof}

%%%%%%%%%%%%%%%%%%%%%%%%%%%%%%%%%%%%%%%%%%%%%%%%%%%%%%%%%%%%%%%%%%
%%%%%%%%     Section 3.2                              %%%%%%%%%%%%%%
%%%%%%%%%%%%%%%%%%%%%%%%%%%%%%%%%%%%%%%%%%%%%%%%%%%%%%%%%%%%%%%%%%

\subsection{The case $f=0$ and $\gamma\neq 0$}\label{SS:f0:gn0}

If $r\neq 1$, then on replacing $h$ by $\tilde{h}=h+\gamma/(1-r)$ we can reduce to the case $\gamma=0$ studied above. So we can assume $r=1$. We can further assume $\gamma=1$, by replacing the generator $h$ by $\gamma^{-1}h$. Let $Q$ be the subalgebra of $L$ generated by $d$ and $u$. Then $Q$ is the quantum plane with relation $du=sud$ and $L=Q[h; \partial]$, where $\partial$ is the derivation of $Q$ determined by $\partial(d)=d$, $\partial(u)=-u$. By the arguments of 
Section~\ref{SS:f0:g0}, $Q$ is a Noetherian UFR. 

\begin{proposition}\label{P:f0:gn0:UFR}
 Assume $f=0$, $r=1$ and $\gamma\neq 0$. Then $L=L(0, 1, s, \gamma)$ is a Noetherian UFR.
\end{proposition}

\begin{proof}
Without loss of generality, we assume $\gamma=1$. By~\cite[Thm.\ 5.5]{CJ86}, it is enough to show that every non-zero $\partial$-prime ideal of $Q$ contains a non-zero principal $\partial$-ideal, for the derivation $\partial$ of $Q$ defined above.  
 
Let $0\neq I\leq Q$ be a $\partial$-prime ideal of $Q$. Choose $p=p(d, u)\in I\setminus \{0\}$ with minimal support, i.e.,  $0\neq p=\sum a_{ij} d^iu^j\in I$ such that $a_{ij}\in\K$ and the set $\{ (i, j)\mid a_{ij}\neq 0 \}$ has minimal cardinality. Fix $(\alpha, \beta)$ such that $a_{\alpha \beta}\neq 0$. It is straightforward to check that $\partial (d^iu^j)=(i-j)d^iu^j$, for all nonnegative integers $i, j$. Then,

\begin{align*}
 I\ni (\alpha-\beta)p-\partial (p)&=\sum a_{ij} (\alpha-\beta)d^iu^j-\sum a_{ij} (i-j)d^iu^j\\
 &=\sum a_{ij} (\alpha+j-\beta-i)d^iu^j.
\end{align*}

\noindent
Furthermore, $(\alpha-\beta)p-\partial (p)$ has a smaller support than $p$, as its coefficient of $d^\alpha u^\beta$ is $0$. By the minimality assumption, it must be that $(\alpha-\beta)p-\partial (p)=0$. Thus, $\partial (p)=(\alpha-\beta)p$. In particular, $i-j$ is constant for all $(i, j)$ such that $a_{ij}\neq 0$, and we can write $p=\sum_{i\geq 0}a_i d^iu^{j(i)}$.

Choose $\alpha$ such that $a_\alpha \neq 0$. Then,

\begin{equation*}
up-s^{-\alpha}pu = \sum_{i\geq 0}a_i (s^{-i}-s^{-\alpha}) d^iu^{j(i)+1}
\end{equation*}

\noindent
and this is still an element of $I$, with smaller support than $p$. Thus, $up-s^{-\alpha}pu =0$ and $up=s^{-\alpha}pu$. Similarly, $dp=s^{\beta}pd$ for some $\beta \in \mathbb{Z}$, which shows that $p$ is normal. In particular, $I$ contains the nonzero principal ideal generated by $p$, which is a $\partial$-ideal, as $\partial (p)=\lambda p$ for some integer $\lambda$. This proves our claim.
 \end{proof}

We now inquire when $L$ is a Noetherian UFD. For that purpose, we will determine the height one prime ideals of $L$ explicitly and check which are completely prime. Since $L$ is a Noetherian UFR, we know that all height one prime ideals are principal. So we start out by determining the normal elements of $L$.

Recall the $\mathbb{Z}$-graduation of $L$ described in~(\ref{E:zgrad}).

\begin{lemma}\label{L:f0:gn0:nor}
Assume $f=0$, $r=1$ and $\gamma\neq 0$. Let $L=L(0, 1, s, \gamma)$ and consider the normal element $z=ud$. 
The normal elements of $L$ are homogeneous and of the form $p(z)u^i$ or $p(z)d^i$, for $i\in\mathbb{N}$ and $p(z)\in\K[z]$. Furthermore:
\begin{enumerate}
\item If $s$ is not a root of unity, then $p(z)=\lambda z^c$ for some $\lambda\in\K$ and some  integer $c\geq 0$;
\item If $s$ is a primitive $l$-th root of unity ($l\geq 1$), then $p(z)=z^c \tilde{p}(z^l)$ for some integer $0\leq c<l$ and some polynomial $\tilde{p}(z^l)$ in the central element $z^l$.
\end{enumerate}

\end{lemma}

\begin{proof}
We again assume $\gamma=1$.
Let $\nu=\sum_{i\in\mathbb{Z}} \nu_i$ be a nonzero normal element of $L$, with $\nu_i$ homogeneous of degree $i$. Notice that $hu^i d^j=u^i d^j (h + j-i)$. In particular, $h \nu_i=\nu_i (h-i)$ for all $i\in\mathbb{Z}$. By degree considerations, there is $\xi (h, z)\in\K[h, z]$ such that $h\nu=\nu \xi (h, z)$. Thus, 

\begin{equation*}
 \sum \nu_i \xi (h, z)= h\sum \nu_i =\sum \nu_i (h-i).
 \end{equation*}

\noindent
Therefore, equating  homogeneous components and factoring out the nonzero $\nu_i$, we get that $h-i=\xi (h, z)$ for all $i$ such that $\nu_i\neq 0$, which proves that $\nu$ is homogeneous.

Assume $\nu\neq 0$ has degree $i\geq 0$. We can write $\nu=p(h, z)u^i$ (see (\ref{E:zgrad:homcomp})). As $u^i d^j$ is clearly normal, for all $i, j\in \mathbb{N}$, and  $L$ is a domain, $p(h, z)u^i$ is normal if and only if $p(h, z)$ is normal. Write $p(h, z)=\sum_{j\geq 0}p_j(h)z^j$, with $p_j (h)\in\K[h]$. As before, there must exist $\xi (h, z)\in\K[h, z]$ such that $up(h, z)=p(h, z) \xi (h, z)u$. Using the commutation relations $up_j (h)=p_j(h+1)u$ and $uz^j=s^{-j}z^j u$, and factoring out $u$ on the right from both terms of that equation, we obtain 

\begin{equation*}
 \sum_{j\geq 0}p_j(h+1)s^{-j}z^j=\xi (h, z)\sum_{j\geq 0}p_j(h)z^j.
\end{equation*}
\noindent
From the above equation we readily conclude that $\xi (h, z)=\xi\in\K$, as we are assuming $p(h, z)\neq 0$. Next, equating coefficients of $z^j$, we get $p_j(h+1)s^{-j}=\xi p_j(h)$ for all $j$. This implies that $p_j(h)$ is a constant polynomial, for all $j$. Thus, we conclude that $\nu=p(z)u^i$, for some $p(z)\in\K[z]$. The case $i\leq 0$ is symmetric.

It remains to determine when  a nonzero element $p(z)\in\K[z]$ is normal. Write $p(z)=\sum_i a_i z^i$, with $a_i\in\K$. Since $up(z)=\sum_i s^{-i}a_i z^i u$, it is easy to conclude that $p(z)$ is normal if and only if there is $\lambda\in\K$ such that $s^{-i}=\lambda$ for all $i$ such that $a_i\neq 0$. Let $c\geq 0$ be the first index for which $a_c\neq 0$. It follows that $p(z)=a_c z^c$ if $s$ is not a root of unity. In case $s$ is a primitive $l$-th root of unity, with $l\geq 1$, then $p(z)=z^c p'(z^l)$, where $p'(z^l)$ is a polynomial in $z^l$, and the result follows.

\end{proof}

We can now list all height one prime ideals of $L$ and check when $L$ is a Noetherian UFD.

\begin{theorem}\label{T:f0:gn0:UFD}
Assume $f=0$, $r=1$ and $\gamma\neq 0$. Let $L=L(0, 1, s, \gamma)$ and  $z=ud$. Then $L$ is a Noetherian UFD if and only if either $s$ is not a root of unity or $s=1$. The height one prime ideals of $L$ are:
\begin{enumerate}
\item $\langle d\rangle$ and $\langle u\rangle$, if $s$ is not a root of unity. These ideals are completely prime.
\item $\langle d\rangle$, $\langle u\rangle$ and $\langle z-\lambda\rangle$, for $\lambda\in\K^*$, if $s=1$. These ideals are completely prime.
\item $\langle d\rangle$, $\langle u\rangle$ and $\langle z^l-\lambda\rangle$, for $\lambda\in\K^*$, if $s$ is a primitive $l$-th root of unity, with $l>1$. The ideals $\langle d\rangle$ and $\langle u\rangle$ are completely prime but those of the form $\langle z^l-\lambda\rangle$ are not.
\end{enumerate}
\end{theorem}

\begin{proof}
Once more, we assume $\gamma=1$.
Let $P$ be a height one prime ideal of $L$. By Proposition~\ref{P:f0:gn0:UFR}, there exists a normal element $\nu$ such that $P=\langle \nu \rangle$. Furthermore, as $P$ is prime, $\nu$ cannot be the product of two nonzero, nonunit  normal elements. Thus, by Lemma~\ref{L:f0:gn0:nor}, the only possibilities are, up to nonzero scalars, $\nu=u$, $\nu=d$ or $\nu=p(z)$. The ideals generated by either $u$ or $d$ are indeed completely prime, as the corresponding factor algebra is isomorphic to the enveloping algebra of the two-dimensional nonabelian Lie algebra. By the Principal Ideal Theorem, they have height one. So it remains to consider the case $\nu=p(z)$. Note first that $z=ud$, and since $u$ and $d$ are nonzero nonunit normal elements, the ideal generated by $z$ is not prime.

Assume first that $s$ is not a root of unity. Then, by Lemma~\ref{L:f0:gn0:nor}(a) and the above, there is no other possibility for $P$. This proves part (a). Now assume $s$ is a primitive $l$-th root of unity, with $l\geq 1$. Since $\K$ is algebraically closed, the only other possibility for the generator $\nu$ of $P$ is $\nu=z^l-\lambda$, for some $\lambda\in\K^{*}$. If $l=1$, i.e., in case $s=1$, then $\langle z-\lambda\rangle$ is completely prime, for $\lambda \neq 0$, as the factor algebra is isomorphic to the differential operator ring $\K[u^{\pm 1}][h;\partial]$, where $\partial (u^{i})=-iu^{i}$, for all $i\in\mathbb{Z}$. This proves part (b). 

Finally, assume $l>1$. Recall that $L$ can be presented as the differential operator ring $Q[h; \partial]$, where  $Q$ is the quantum plane with relation $du=sud$ and  $\partial$ is the derivation of $Q$ determined by $\partial(d)=d$, $\partial(u)=-u$. The centre of $Q$ is the polynomial algebra $\K[d^l, u^l]$, by~\cite[1.3(i)]{GL98}, and the element $z^l-\lambda$ is  irreducible  in this polynomial algebra. Thus, $z^l-\lambda$ generates a prime ideal of $\K[d^l, u^l]$. By~\cite[Cor.\ 1.5]{GL98}, $z^l-\lambda$ also generates a prime ideal of $Q$. Furthermore, $\partial (z^l-\lambda)=0$, as $\partial (z)=0$. Hence, by~\cite[Prop.\ 14.2.5]{MR01}, $z^l-\lambda$ generates a prime ideal of $L$. This ideal has height one, by the Principal Ideal Theorem. Since $l>1$ and $\K$ is algebraically closed, $z^l-\lambda$ factors nontrivially as a polynomial in $z$, so the ideal $\langle z^l-\lambda\rangle$ is not completely prime, which proves part (c).
\end{proof}

%%%%%%%%%%%%%%%%%%%%%%%%%%%%%%%%%%%%%%%%%%%%%%%%%%%%%%%%%%%%%%%%%%
%%%%%%%%     Section 4                              %%%%%%%%%%%%%%
%%%%%%%%%%%%%%%%%%%%%%%%%%%%%%%%%%%%%%%%%%%%%%%%%%%%%%%%%%%%%%%%%%

\section{The case $f$ conformal and $r$ not a root of unity}\label{S:rnru}

In this case, as $r\neq 1$, we can and will assume that $\gamma=0$. Thus, the defining relations for $L=L(f, r, s, 0)$ are:
\begin{align}\label{E:g0:defGdua1}
dh-rhd&=0,\\\label{E:g0:defGdua2}
hu-ruh&=0,\\\label{E:g0:defGdua3}
du-sud+f(h)&=0.
\end{align}

\noindent
In particular, $h$ is a nonzero, nonunit normal element of $L$ which generates a completely prime ideal of $L$. 

Let $L_h$ be the localisation of $L$ with respect to the powers of $h$. It is clear that 
\begin{equation*}
 L_h=\K[h^{\pm 1}][d;\sigma][u; \sigma^{-1}, \delta],
\end{equation*}
where $\sigma$ and $\delta$ are extended by setting $\sigma(h^{-1})=r^{-1}h^{-1}$ and $\delta (h^{-1})=0$.

\begin{proposition}\label{L:rnru:iffUFR}
Assume $\gamma =0$. Then $L_h$ is a Noetherian UFR (resp.\ UFD) if and only if $L$ is a Noetherian UFR (resp.\ UFD). 
\end{proposition}

\begin{proof}
The direct implication follows from Lemma~\ref{nagata}. The converse follows from standard arguments in localisation theory, provided we can show that any normal element of $L$ is still normal in $L_h$.

Let $\nu\in L$ be normal. We can assume $\nu\neq 0$. Write $\nu=\sum_{i\in \mathbb{Z}} \nu_i$ with $\nu_i$ homogeneous of degree $i$. As 
$h\nu_i=\nu_i \sigma^i (h)=r^i\nu_i h$, it follows that $h\nu=\sum_{i\in \mathbb{Z}}r^i\nu_i h$. On the other hand, by the normality of $\nu$, there is $h'\in L$ such that $h\nu=\nu h'$. The $\mathbb{Z}$-grading implies that $h'$ has degree $0$. Hence, the degree $i$ component of $\nu h'$ is $\nu_i h'$. Equating elements of the same degree we conclude that $r^i\nu_i h=\nu_i h'$, for all $i$. Choose $\alpha$ with $\nu_{\alpha}\neq 0$. We must have $h'=r^{\alpha}h$ and  thus $h\nu=r^{\alpha}\nu h$. Hence, in $L_h$, we have $\nu h^{-1}=r^{\alpha}h^{-1}\nu$, which is enough to show that $\nu$ is normal in $L_h$. This concludes the proof.
\end{proof}

Let $g\in\K[h]$ be such that $f=g^*$, i.e., $f(h)=sg(h)-g(rh)$. We will assume $f\neq 0$, as the case $f=0$ has already been dealt with in Section~\ref{S:f0}. In particular, $g\neq 0$.

The algebra $L_h$ is in the scope of the algebras studied by Jordan in~\cite{dJ93}, where our polynomial $g$ plays the role of the element $u$ of~\cite{dJ93}. We will start out with a few technical observations which will allow us to apply the results of~\cite{dJ93} to $L_h$. We will remind the reader of the necessary definitions as they are needed.

\begin{lemma}\label{L:rnru:fc:ssimple} Assume $\gamma=0$ and $r$ is not a root of unity. Then
the Laurent polynomial algebra $\K[h^{\pm 1}]$ is $\sigma$-simple, i.e., its only $\sigma$-invariant ideals are itself and  the zero ideal.  
\end{lemma}

\begin{proof}
 This is worked out in Example 1.2.(i) of~\cite{dJ93}.
\end{proof}

We recall Definition 1.7 of~\cite{dJ93}, applied in our context. Suppose there exists $0\neq p\in \K[h^{\pm 1}]$ such that $\sigma (p)=s^{-n}p$ for some positive integer $n$. Let $n\geq 1$ be minimal with respect to the existence of such  an element $p$. Then, any $0\neq p\in \K[h^{\pm 1}]$ satisfying $\sigma (p)=s^{-n}p$ will be called a {\it principal eigenvector}, and $n$ will be its {\it degree}.

In order to discuss the existence of principal eigenvectors, we will make use of two integers  $\epsilon\in\mathbb{Z}$ and $\tau\in\mathbb{N}$, which have been defined in~\cite{CL09}, as follows: 
\begin{equation*}
\tau = \min \{ i>0\mid s^{i}=r^{j}\quad \text{for some $j\in\mathbb{Z}$} \}\quad\quad \text{and}\quad\quad r^{\epsilon}=s^{\tau},
\end{equation*}
if $\{ i>0\mid s^{i}=r^{j}\quad \text{for some $j\in\mathbb{Z}$} \}\neq\emptyset $, and
$\tau= 0=\epsilon$, otherwise.
As long as $r$ is not a root of unity, $\epsilon$ is uniquely defined.
Furthermore, by~\cite[Lem.\ 2.1]{CL09}, if $\delta, \eta\in\mathbb{Z}$ then $r^\delta s^\eta=1$ if and only if there is $\lambda\in\mathbb{Z}$ such that $(\delta, \eta)=\lambda (-\epsilon, \tau)$.

\begin{lemma} \label{L:rnru:fc:peigen} 
Assume $\gamma=0$. 
There exist principal eigenvectors in $L_h$ if and only if $\tau>0$, i.e., if and only if there are integers $\alpha, \beta$, with $\alpha\neq 0$, such that $s^\alpha r^\beta =1$. 
\end{lemma}

\begin{proof}
Assume  $s^\alpha r^\beta =1$, with $\alpha\neq 0$. We can thus assume $\alpha\geq 1$. Then $\sigma (h^\beta)=r^\beta h^\beta=s^{-\alpha}h^\beta$, so there are principal eigenvectors. 

Conversely, assume $\sigma (p)=s^{-n}p$ for some $n\geq 1$ and some  $p=\sum_{i=\alpha}^{\beta}a_i h^i$, with $\alpha\leq\beta$ and $a_\alpha a_\beta \neq 0$. Then, $0=\sigma (p)-s^{-n}p=\sum_{i=\alpha}^{\beta}a_i (r^i - s^{-n}) h^i$. In particular, $r^\beta - s^{-n}=0$ and $s^n r^\beta=1$.
\end{proof}

\begin{theorem} \label{T:rnru:fc:UFR}
Assume $\gamma=0$, $r$ is not a root of unity and $f\neq 0$ is conformal. Then $L$ is a Noetherian UFR if and only if either $\tau>0$ or $f$ is a monomial.
\end{theorem}

\begin{proof}
By Proposition~\ref{L:rnru:iffUFR}, we can work over the localisation $L_h$. Then, the result for $L_h$ follows from Example 2.21 of~\cite{dJ93}.
\end{proof}

It remains to establish when $L$ is a Noetherian UFD, which we do next.

\begin{theorem} \label{T:rnru:fc:UFD}
Assume $\gamma=0$, $r$ is not a root of unity and $f\neq 0$ is conformal. Then $L$ is a Noetherian UFD if and only if either one of the following two conditions holds:
\begin{enumerate}
\item $\langle r, s \rangle$ is a free abelian group of rank $2$ and $f$ is a monomial, or
\item $\langle r, s \rangle$ is a free abelian group of rank $1$.
\end{enumerate}
\end{theorem}

\begin{proof}
Once again, we can work over the localisation $L_h$, by virtue of Proposition~\ref{L:rnru:iffUFR}. Notice that, since $r$ is not a root of unity, $\langle r, s \rangle$ is a free abelian group of rank $2$ if and only if $\tau=0$, and $\langle r, s \rangle$ is a free abelian group of rank $1$ if and only if $\tau\geq 1$ and $\gcd(\tau, \epsilon)=1$. 

Assume first that $\tau=0$. Then, by the above, $L_h$ is a Noetherian UFR if and only if  $f$ is a monomial. When this is the case, $\langle z:=ud-g(h) \rangle$ is the unique height one prime ideal of $L_h$, and it is clearly completely prime, as the factor algebra is a quantum torus in two variables (namely, the cosets of $u$ and $h$). 

Let us now suppose $\tau\geq 1$. Then, by the proof of Lemma~\ref{L:rnru:fc:peigen}, there is a principal eigenvector, $h^{-\epsilon}$, and it has degree $\tau$. Furthermore, this principal eigenvector is unique, up to nonzero scalar multiples. It follows that $\langle h^{-\epsilon}z^{\tau}-\lambda \rangle$   is a height one prime ideal of $L_h$, for all $\lambda\in\K^*$ such that $\lambda h^\epsilon\neq (-g(h))^\tau$, by~\cite[Cor.\ 2.9.(ii)]{dJ93}. Note that $\lambda h^\epsilon= (-g(h))^\tau$ can occur for at most one value of $\lambda\in\K^*$. By~\cite[Thm.\ 2.24]{dJ93}, it is easy to conclude that, for $\lambda\in\K^*$, $\langle h^{-\epsilon}z^{\tau}-\lambda \rangle$ is completely prime if and only if $\gcd(\tau, \epsilon)=1$. In particular, under the current hypotheses, if  $L_h$ is a Noetherian UFD, then $\langle r, s \rangle$ is free abelian of rank $1$.%$\gcd(\tau, \epsilon)=1$. 

Conversely, assume $\gcd(\tau, \epsilon)=1$, with $\tau\geq 1$. Then, by~\cite[2.17 and Prop.\ 2.18]{dJ93}, the height one prime ideals of $L_h$ include $\langle z \rangle$, which is completely prime, and the  ideals of the form $\langle h^{-\epsilon}z^{\tau}-\lambda \rangle$, for  $\lambda\in\K^*$ such that $\lambda h^\epsilon\neq (-g(h))^\tau$, which are all completely prime, as $\gcd(\tau, \epsilon)=1$. In case  $h^{-\epsilon}(-g(h))^\tau\notin\K$, then this is the complete list of height one prime ideals of $L_h$, and it follows that $L_h$ is a Noetherian UFD. Suppose that $h^{-\epsilon}(-g(h))^\tau\in\K$. Then $g(h)$ is a unit, say $g(h)=\mu h^a$, and it follows that $\epsilon=\tau a$. As we are assuming $\tau$ and $\epsilon$ to be coprime, it must be that $\tau=1$ and $\epsilon=a$. Thus, $f(h)=sg(h)-g(rh)=\mu (s-r^a)h^a=\mu (s^\tau-r^\epsilon)h^a=0$, which contradicts our hypothesis on $f$. Therefore, it is always the case that  $h^{-\epsilon}(-g(h))^\tau\notin\K$ and the proof is complete.
\end{proof}

%%%%%%%%%%%%%%%%%%%%%%%%%%%%%%%%%%%%%%%%%%%%%%%%%%%%%%%%%%%%%%%%%%
%%%%%%%%     Section 5                              %%%%%%%%%%%%%%
%%%%%%%%%%%%%%%%%%%%%%%%%%%%%%%%%%%%%%%%%%%%%%%%%%%%%%%%%%%%%%%%%%

\section{The case $f$ conformal and $r=1$}\label{S:r1}

When $r=1$, we cannot assume that $\gamma=0$, so we will consider separately the cases $\gamma\neq 0$ and $\gamma=0$. The defining relations of $L=L(f, 1, s, \gamma)$ are: 
\begin{align}\label{E:r1:defGdua1}
dh-hd+\gamma d&=0,\\\label{E:r1:defGdua2}
hu-uh+\gamma u&=0,\\\label{E:r1:defGdua3}
du-sud+f(h)&=0.
\end{align}
Note that if $r=1$ and $\gamma\neq 0$ we retrieve the algebras studied by Rueda in~\cite{sR02}. The latter include Smith's algebras~\cite{spS90}, which occur as generalized down-up algebras when $r=s=1$ and $\gamma\neq 0$. We assume throughout that $f\neq 0$.

%%%%%%%%%%%%%%%%%%%%%%%%%%%%%%%%%%%%%%%%%%%%%%%%%%%%%%%%%%%%%%%%%%
%%%%%%%%     Section 5.1                              %%%%%%%%%%%%%%
%%%%%%%%%%%%%%%%%%%%%%%%%%%%%%%%%%%%%%%%%%%%%%%%%%%%%%%%%%%%%%%%%%
\subsection{The case $f$ conformal, $r=1$ and $\gamma\neq 0$}\label{SS:r1:gn0}

Let $g$ be such that $f(h)=sg(h)-g(h-\gamma)$. In particular, $g\neq 0$. Recall, from Section~\ref{S:rnru}, the definition of a principal eigenvector.

\begin{lemma}\label{L:r1:gno:eig} 
Assume $r=1$ and $\gamma\neq0$. If $p\in\K[h]$ is such that $\sigma (p)=\mu p$ for some $\mu\in\K$ then $p\in\K$. 
In particular, the only nonzero $\sigma$-invariant ideal of $\K[h]$ is $\K[h]$ and there are principal eigenvectors if and only if $s$ is a root of unity.  
\end{lemma}
\begin{proof}
Suppose that $\sigma (p)=\mu p$ for some $p\in\K[h]\setminus\K$. Then, since $\K$ is algebraically closed, there is $\alpha\in\K$ such that $p(\alpha)=0$. It follows that
$
 0=\mu p(\alpha)
 =\sigma(p)(\alpha)
 =p(\alpha-\gamma),
$
and hence $\alpha-\gamma$ is also a root of $p$. Since $\alpha$ was an arbitrary root of $p$ and $\gamma\neq 0$, this is impossible. Thus, $p\in\K$.

Let $I$ be a  $\sigma$-invariant ideal of $\K[h]$. Then $I=\langle p\rangle$, for some $p\in\K[h]$, and $\sigma(p)\in\K^{*}p$, so either $I=\langle 0\rangle$ or $I=\K[h]$.

Finally, assume there is a principal eigenvector $0\neq p\in\K[h]$. Then there is $n\geq 1$ so that $\sigma (p)=s^{-n}p$. In particular, by the above, it follows that $p\in\K^*$ and $s^n=1$. Conversely, if $s$ is a primitive $n$-th root of unity, then $1$ is a principal eigenvector of degree $n$.
\end{proof}

\begin{proposition}\label{P:r1:gno:propH}
Assume $s$ is not a root of unity and $\gamma\neq 0$. Take $p\in\K[h]$. Then $p$ satisfies
\begin{equation}\label{E:r1:gno:propH}
\forall \lambda\in\K \ \ \forall n\geq 1\quad  s^{n}p(\lambda)=p(\lambda-n \gamma) \implies p(\lambda)=0
\end{equation}
if and only if $p\in\K$.
 \end{proposition}

\begin{proof}
Let  $p\in\K[h]$ and assume, by way of contradiction, that $p$ is not constant. Then the set of roots of $p$ is finite and nonempty. Let $\Delta=\{ \alpha-\beta\mid \alpha \ \mbox{and} \ \beta \ \mbox{are roots of}\  p \}$ be the set of differences of (not necessarily distinct) roots of $p$. Since $\Delta$ is finite, there exists an integer $n\geq 1$ such that $n\gamma\notin\Delta$. Consider the polynomial $p_n (h)=s^np(h)-p(h-n\gamma)$. Since $s$ is not a root of unity, $p_n$ has the same degree as $p$. In particular, $p_n$ has some root, say $\alpha\in\K$. 
By~(\ref{E:r1:gno:propH}), $\alpha$ is a root of $p$, which in turn implies that $\alpha-n\gamma$ is a root of $p$, as well. Hence, $n\gamma=\alpha-\left( \alpha-n\gamma \right)\in \Delta$, which contradicts our choice of $n$. So indeed, $p\in\K$. The converse implication is trivial.
\end{proof}

We are now ready to say when $L$ is a Noetherian UFR.

\begin{theorem}\label{T:r1:gno:UFR}
Assume $f$ and $\gamma$ are nonzero and $r=1$. Then $L=L(f, 1, s, \gamma)$ is a Noetherian UFR if and only if one of the following conditions hold:
\begin{enumerate}
\item $s$ is a root of unity, or
\item $s$ is not a root of unity and $f\in\K$.
\end{enumerate}
 \end{theorem}
 
\begin{proof}
 The first and second parts follow from~\cite[Prop.\ 2.18]{dJ93} and~\cite[Prop.\ 2.20]{dJ93}, respectively, and the results in this section. Note that, since $s\neq 1$, then $f\in\K\iff g\in\K$.
\end{proof}

Next, we deduce from~\cite[2.22 and Remark\ 2.25]{dJ93} the cases where $L$ is a Noetherian UFD.

\begin{theorem}\label{T:r1:gno:UFD}
Assume $f$ and $\gamma$ are nonzero and $r=1$. Then $L=L(f, 1, s, \gamma)$ is a Noetherian UFD if and only if one of the following conditions hold:
\begin{enumerate}
\item $s=1$, or
\item $s$ is not a root of unity and $f\in\K$.
\end{enumerate}
 \end{theorem}

%%%%%%%%%%%%%%%%%%%%%%%%%%%%%%%%%%%%%%%%%%%%%%%%%%%%%%%%%%%%%%%%%%
%%%%%%%%     Section 5.2                              %%%%%%%%%%%%%%
%%%%%%%%%%%%%%%%%%%%%%%%%%%%%%%%%%%%%%%%%%%%%%%%%%%%%%%%%%%%%%%%%%
\subsection{The case $f$ conformal, $r=1$ and $\gamma= 0$}\label{SS:r1:g0}

In this case, $h$ is central in $L$. Note also that since $f$ is conformal in $L$, then $s\neq1$. Nevertheless,  the conformality condition will not be used in this section, as we will not refer to~\cite{dJ93}. We will consider separately the cases $s$ not a root of unity and $s$ a root of unity of order $l\geq 2$. 

%%%%%%%%%%%%%%%%%%%%%%%%%%%%%%%%%%%%%%%%%%%%%%%%%%%%%%%%%%%%%%%%%%
%%%%%%%%     Section 5.2.1                              %%%%%%%%%%%%%%
%%%%%%%%%%%%%%%%%%%%%%%%%%%%%%%%%%%%%%%%%%%%%%%%%%%%%%%%%

\subsubsection{The case $f$ conformal, $r=1$, $\gamma=0$ and $s$ not a root of unity}\label{SS:r1:g0:snru}

\begin{theorem}\label{T:r1:g0:snru:UFD}
 Assume $f$ is nonzero, $r=1$, $\gamma=0$ and $s$ is not a root of unity. Then $L=L(f, 1, s, 0)$ is a Noetherian UFD. In fact, the height  one prime ideals of $L$ are $\langle h-\lambda \rangle$, for $\lambda\in\K$, and $\langle du-ud \rangle$.
\end{theorem}

\begin{proof}
 Let $\lambda\in\K$. Then $h-\lambda$ is central and the factor algebra $L/\langle h-\lambda \rangle$ is either the quantum plane or the quantum Weyl algebra, depending on whether $\lambda$ is a root of $f$ or not. In either case, 
$L/\langle h-\lambda \rangle$ is a domain, and thus the ideal $\langle h-\lambda \rangle$ is completely prime.

The element $du-ud=(s-1)z$ is normal and the factor algebra $L/\langle du-ud \rangle$ is a commutative algebra generated by $h$, $d$ and $u$, subject to the relation $ud-\frac{1}{s-1}f(h)=0$. It is easy to see that the element $ud-\frac{1}{s-1}f(h)$, viewed as an element of the polynomial algebra in the $3$ commuting variables $h$, $d$ and $u$, is irreducible, provided that $f\neq 0$. This shows that $L/\langle du-ud \rangle$ is a domain. Another way of reaching this conclusion is by realising this factor algebra as the generalized Weyl algebra $\K[h]\left( \mathrm{id}_{\K[h]}, \frac{1}{s-1}f(h) \right)$. (The reader is referred to \cite{Bav93} for more details on generalized Weyl algebras.)

The above shows that all ideals of the form $\langle h-\lambda \rangle$, for $\lambda\in\K$, and $\langle du-ud \rangle$ are completely prime and principal. By the Principal Ideal Theorem, they have height one. To finish the proof, we need only show that any nonzero prime ideal of $L$ must contain one of these ideals. We do so in the next proposition.
\end{proof}

\begin{proposition}\label{P:r1:g0:snru:P}
 Assume $f$ is nonzero, $r=1$, $\gamma=0$ and $s$ is not a root of unity. Then any nonzero prime ideal of $L=L(f, 1, s, 0)$ must contain either $du-ud$ or $h-\lambda$, for some $\lambda\in\K$.
\end{proposition}

\begin{proof}
 Let $P$ be a nonzero prime ideal of $L$ and assume $h-\lambda$ is not in $P$, for any scalar $\lambda$. We will show that $du-ud\in P$.
 
 Since 
 $\K$ is algebraically closed and $\K[h]$ is a central subalgebra, it follows that $P\cap\K[h]=\langle 0\rangle $. Let 
 $\widetilde{L}$ be the localisation of $L$ at the central multiplicative set of nonzero elements of $\K[h]$. Let $\mathbb{F}=\K(h)$ be the field of fractions of $\K[h]$. Then $\widetilde{L}$ can be seen as the first quantised Weyl algebra $\mathbb{A}^{s}_{1}(\mathbb{F})$, generated over $\mathbb{F}$ by $X$ and $Y$, and subject to the relation $XY-sYX=1$. In fact, it is easy to check that there  are mutually inverse $\mathbb{F}$-algebra  maps, 
 $\Phi:\widetilde{L}\rightarrow \mathbb{A}^{s}_{1}(\mathbb{F})$ and $\Psi:\mathbb{A}^{s}_{1}(\mathbb{F})\rightarrow \widetilde{L}$, such that $\Phi (d)=X$, $\Phi (u)=-f(h)Y$, $\Psi (X)=d$ and $\Psi (Y)=-u\left( f(h) \right)^{-1}$.
 
 Thus, $P$ extends to a nonzero prime ideal $\widetilde{P}$ of $\widetilde{L}$, which we identify, via the map $\Phi$ above, with a prime ideal of $\mathbb{A}^{s}_{1}(\mathbb{F})$. The element $Z:=XY-YX$ of $\mathbb{A}^{s}_{1}(\mathbb{F})$ is normal, nonzero and not a unit. In fact, $Z$ corresponds, via $\Psi$, to the element $(ud-du)\left( f(h) \right)^{-1}$. Let 
 $\mathbb{B}^{s}_{1}(\mathbb{F})$ be the localisation of $\mathbb{A}^{s}_{1}(\mathbb{F})$ at the powers of $Z$. Since $s$ is not a root of unity, $\mathbb{B}^{s}_{1}(\mathbb{F})$ is simple, by~\cite[Lem.\ 2.2]{AVV88} (note that this result does not depend on the base field being algebraically closed). Since $Z$ is normal, this means that every nonzero prime ideal of $\mathbb{A}^{s}_{1}(\mathbb{F})$ contains $Z$. In particular, $Z\in \Phi (\widetilde{P})$, i.e., $(ud-du)\left( f(h) \right)^{-1}\in \widetilde{P}$. Thus, $du-ud\in P=\widetilde{P}\cap L$.
\end{proof}

%%%%%%%%%%%%%%%%%%%%%%%%%%%%%%%%%%%%%%%%%%%%%%%%%%%%%%%%%%%%%%%%%%%
%%%%%%%%%     Section 5.2.2                              %%%%%%%%%%%%%%
%%%%%%%%%%%%%%%%%%%%%%%%%%%%%%%%%%%%%%%%%%%%%%%%%%%%%%%%%%

\subsubsection{The case $f$ conformal, $r=1$, $\gamma=0$ and $s\neq 1$  a root of unity}\label{SS:r1:g0:sru}

We finally tackle the case in which $s$ is a primitive $l$-th root of unity, for some $l\geq 2$. It is straightforward to see that, in this case, both $d^l$ and $u^l$ are central.
Our aim is to prove that, in this case, $L$ is a Noetherian UFR. 

Let $\widetilde{L}$ be the localisation of $L$ with respect to the multiplicative set generated by the central elements of the form $h-\lambda$, where $\lambda$ runs through the roots of $f$. In case $f$ is a (nonzero) constant polynomial, we have $\widetilde{L}=L$.

Since, for $\lambda$ a root of $f$, $L/\langle h-\lambda \rangle$ is a quantum plane, the ideals of the form  $\langle h-\lambda \rangle$, with $f(\lambda)=0$, are completely prime as well as pairwise distinct. Thus, by~\cite[Prop.\ 1.6]{LLR06}, it will be enough to show that $\widetilde{L}$ is a Noetherian UFR.

Let $S$ be the localisation of $\K[h]$ at the multiplicative set generated by the $h-\lambda$, with $\lambda$ running through the roots of $f$. Since $\K$ is algebraically closed, $f$ is a product of linear factors, and thus is invertible in $S$. The localised algebra $\widetilde{L}$ can be seen as the algebra over  $S$, generated by elements $D$ and $U$, subject to the relation
\begin{equation*}
DU-sUD=1, 
\end{equation*}
where $D=d$ and $U=-u\left( f(h) \right)^{-1}$. Consider the nonzero normal element $Z=DU-UD$, of $\widetilde{L}$. It satisfies  $ZU=sUZ$ and $ZD=s^{-1}DZ$. In particular, $Z^l$ is central in $\widetilde{L}$. The algebra 
$\widetilde{L}/\langle Z \rangle$ is isomorphic to the commutative Laurent polynomial algebra $S[U^{\pm 1}]$, and hence $Z$ generates a completely prime ideal of $\widetilde{L}$. Therefore, it will suffice to show that the localisation $\widehat{L}$ of $\widetilde{L}$ at the multiplicative set generated by $Z$ is a Noetherian UFR, by~\cite[Prop.\ 1.6]{LLR06}. The latter is a consequence of the result that follows.

\begin{proposition}\label{P:r1:g0:sru:Az}
Under the above assumptions, $\widehat{L}$ is an Azumaya algebra over its centre $\Z(\widehat{L})$, with $[\widehat{L}:\Z(\widehat{L})]=l^2$. Moreover, the centre $\Z(\widehat{L})$ of $\widehat{L}$ is the localisation of $S[U^l, D^l]$ at the powers of $Z^l$.
\end{proposition}

\begin{proof}
 The proof is entirely analogous to that of~\cite[Prop.\ 1.3]{AD96}. We give details for completeness.
 
 First, it is easy to see that the centre of $\widetilde{L}$ is $S[U^l, D^l]$, and it must contain $Z^{l}$, as this element commutes with $D$ and $U$.
% $$Z^l=s^{l(l-1)/2}\left[ (s-1)^l U^lY^l-(-1)^l\right] \in \Z(\widetilde{L}).$$
 
 Let $b=a Z^{n} $ be an element of $\widehat{L}$ with $a \in \widetilde{L}$ and $n \in \mathbb{Z}$. Take $q \in \mathbb{Z}$ and $0 \leq r <l$ such that $n=ql+r$. As $Z^{ql}$ is central in $\widehat{L}$, we get that $b=a Z^{ql}Z^r$ is central in $\widehat{L}$ if and only if $a Z^r$ is central in  $\widetilde{L}$. Hence, $\Z(\widehat{L})=\{c Z^{ql} ~|~c\in S[U^l,D^l],~q\in \mathbb{Z}\}$ is the localisation of $S[U^l,D^l]$ at the powers of $Z^l$. 
 
By~\cite[Lem.\ 1.2]{AD96}, $\left\{ U^iD^j \right\}_{0\leq i, j\leq l-1}$ is a basis for $\widehat{L}$ over its centre. So $[\widehat{L}:\Z(\widehat{L})]=l^2$.

To conclude, it is enough to show that the irreducible finite dimensional representations of $\widehat{L}$ over $\K$ all have dimension $l$, by the Artin-Procesi Theorem. Let $\rho : \widehat{L}\rightarrow \mathrm{End}_{\K}(V)$ be such an irreducible representation, with $\dim_{\K} V=m$. Since $\K$ is algebraically closed, and $V$ is finite-dimensional, it follows by Schur's Lemma that the centre of $\widehat{L}$ acts on $V$ as scalars. Thus, $\dim_{\K} \rho({\widehat{L}})\leq l^2$. By the Jacobson Density Theorem, $\rho$ is surjective. Therefore,
$$
m^2=\dim_{\K} \mathrm{End}_{\K}(V)=\dim_{\K} \rho({\widehat{L}})\leq l^2,
$$
and $m\leq l$.

On the other hand, let $X=\rho (D)$, $Y=\rho (U)$. Then, $XY-sYX=\rho (DU-sUD)=\rho(1)=1\in\mathrm{End}_{\K}(V)$. Furthermore, as $Z$ is invertible in $\widehat{L}$, the same is true of $\rho(Z)=XY-YX\in\mathrm{End}_{\K}(V)$. Thus, $\dim_{\K} \mathrm{End}_{\K}(V)\geq l^2$, again by~\cite[Lem.\ 1.2]{AD96}. So $m\geq l$ and $m=l$. 
\end{proof}

\begin{theorem}\label{T:r1:g0:sru:UFR}
 Assume $f$ is nonzero, $\gamma=0$, $r=1$ and $s$ is a primitive $l$-th root of unity, for some $l\geq 2$. Then $L=L(f, 1, s, 0)$ is a Noetherian UFR, but not a Noetherian UFD.
\end{theorem}

\begin{proof}
Since the algebra $\widehat{L}$ is Azumaya over its centre, it follows that all ideals of $\widehat{L}$ are centrally generated. Hence, as  $\Z(\widehat{L})$ is a (commutative) UFR, by Proposition~\ref{P:r1:g0:sru:Az}, we deduce from the Principal Ideal Theorem that $\widehat{L}$ is a Noetherian UFR. We can thus conclude that $L$ is a Noetherian UFR, by~\cite[Prop.\ 1.6]{LLR06}. 

We will now observe that the ideal $\langle d^l-1\rangle$ of $L$ is prime. To see this, notice that $L$ is an Ore extension over the commutative polynomial algebra $\K[h, d]$. So, by~\cite[Prop.\ 2.1]{aB85}, it will be enough to prove that $d^l-1$ generates a $\delta$-stable, $\sigma$-prime ideal of this polynomial algebra, where $\sigma$ and $\delta$ are as in~(\ref{E:LitOre}). In particular,  $\sigma(d)=sd$. This ideal is stable under $\delta$ and $\sigma$ because $d^l-1$ is central in $L=\K[h, d][u;\sigma^{-1}, \delta]$. Consider the prime ideal $I$ of $\K[h, d]$ generated by $d-1$. Since $s$ is a primitive root of unity of order $l$, it follows that 
$$
\bigcap_{i\in\mathbb{Z}} \sigma^{i}(I)=\prod_{0\leq i\leq l-1} (d-s^{i})\K[h, d]=(d^{l}-1)\K[h, d],
$$
so $(d^{l}-1)\K[h, d]$ is indeed a $\sigma$-prime ideal of $\K[h, d]$, as it is the intersection of a $\sigma$-orbit of a prime ideal. Thus, $\langle d^l-1\rangle$ is a prime ideal of $L$.

By the Principal Ideal Theorem,  $\langle d^l-1\rangle$  has height one. Yet,  it is not completely prime, as   $l\geq 2$ and hence the central element  $d^l-1$ factors non-trivially. So $L$ is not a Noetherian UFD.
\end{proof}

%%%%%%%%%%%%%%%%%%%%%%%%%%%%%%%%%%%%%%%%%%%%%%%%%%%%%%%%%%%%%%%%%%
%%%%%%%%     Section 6                              %%%%%%%%%%%%%%
%%%%%%%%%%%%%%%%%%%%%%%%%%%%%%%%%%%%%%%%%%%%%%%%%%%%%%%%%%%%%%%%%%

\section{The case $f$ conformal and $r\neq 1$ a  root of unity}\label{S:rr1}

The final part of our discussion concerns the case when $f$ is conformal and $r$ is a primitive root of unity of order $l\geq 2$. Since $r\neq 1$ we will assume, without loss of generality, that $\gamma=0$, by Proposition~\ref{P:conf:gamma0rnot1}.

We start with a negative result, which follows immediately from Lemma~\ref{L:ql:h1p}.

\begin{corollary}\label{C:rr1:notUFD}
Let $L=L(f, r, s, 0)$ and assume  $r\neq 1$ is a root of unity. Then $L$ is not a Noetherian UFD.
\end{corollary}

The remainder of this section is devoted to establishing that, under the current assumptions, 
 $L=L(f, r, s, 0)$ is a Noetherian UFR. The following general result will play, in this section, the role of Propositions~\ref{P:locX} and~\ref{P:locY}.

We consider a Noetherian ring  $R$, with a subring $A$, which is a domain, and such that $R$ is free both as a left and as a right  $A$-module, with basis $S:=\{X^i ~|~i\geq 0\}$. Assume the multiplicative system $S$ satisfies the Ore condition on both sides, and let $\widehat{R}:=RS^{-1}$ be the corresponding localisation.

\begin{lemma}\label{L:rr1:pi}
Let $P$ be a  nonzero prime ideal  of $R$ such that $P \cap S =\emptyset$, and assume that there exists $b \in \widehat{R}$ such that:
\begin{enumerate}
\item $PS^{-1}=\widehat{R}b=b\widehat{R}$;
\item $Xb=\eta bX$, for some central unit $\eta$ of $A$.
\end{enumerate}

Then $P=xR=Rx$, where $e \in \mathbb{Z}$ is minimal such that $bX^e \in R$, and $x=bX^e$.
\end{lemma}

\begin{proof}
Observe that, since $b\neq 0$, a minimal $e \in \mathbb{Z}$ such that $bX^e \in R$ exists; also, $Xx=\eta xX$. We will prove that $P=Rx$. As $Xb=\eta bX$, $e$ is also minimal such that $X^e b \in R$, and $X^e b=\eta^{e}x$, so a similar argument will show that $P=xR$, using the fact that $\{X^i ~|~i\geq 0\}$ is  a free basis for $R$ as a right $A$-module.

By construction, it is clear that $Rx \subseteq P$, as $x\in PS^{-1}\cap R=P$. Let $y \in P$. Then $y \in PS^{-1}=\widehat{R}b=\widehat{R}xX^{-e}= \widehat{R}x$, as $x$ and $X$ $\eta$-commute. Hence, there exists $u \in \widehat{R}$ such that 
$y=ux$. Moreover, there exists $t \geq 0$ such that $u X^t \in R$. Therefore, $y X^t=uxX^t=\eta^{-t} uX^t x$, i.e., there exist $t \geq 0$ and $r \in R$ such that 
$yX^t=rx$. We choose a minimal such  $t$.

Assume that $t \geq 1$. Write 
$$r=\sum_{i=0}^k r_i X^i, ~~ y=\sum_{i=0}^k y_i X^i,~~x=\sum_{i=0}^k x_i X^i,$$
where $r_i,y_i,x_i \in A$. Note that $x_0 \neq 0$, as otherwise $xX^{-1} \in R$, so that $bX^{e-1} \in R$, contradicting the minimality of $e$. 

On the other hand, as $Xx=\eta xX$, the equality $yX^t=rx$ can be written as follows:
$$
 \sum_{i=0}^ky_iX^{i+t}=\sum_{i=0}^k r_i X^i bX^e  
 =\sum_{i=0}^k r_i \eta^i bX^{e+i}  
 =\sum_{i=0}^k r_i \eta^i xX^{i}  
 =\sum_{i,j=0}^k r_i\eta^ix_jX^{i+j}.
$$
As $t \geq 1$, identifying the degree 0 coefficients yields $0=r_0x_0$. As $x_0 \neq 0$ and $A$ is a domain, this forces $r_0=0$. 
Hence, $rX^{-1} \in R$ and $yX^{t-1}=rxX^{-1}=\eta rX^{-1}x$. This contradicts the minimality of $t$. Thus, $t=0$ and $y =rx \in Rx$, as desired.
\end{proof}

\begin{proposition}\label{P:rr1:udnP}
Let $L=L(f, r, s, 0)$, with $f$  conformal. If $P$ is a prime ideal of $L$ of height one, which either does not contain any power of $d$ or does not contain any power of $u$, then $P$ is a principal ideal, generated by a normal element of $L$.
\end{proposition}

\begin{proof}
By Lemma~\ref{L:conf:locd}, the localisation $\widehat{L}$ of $L$ at the denominator set $D=\{ d^{i}\}_{i\geq 0}$ is isomorphic to a quantum coordinate ring of affine $3$-space over $\K$, localised at the powers of one of its canonical generators. As in Section~\ref{SS:f0:g0}, it follows that $\widehat{L}$ is a Noetherian UFR.

If $P$ is a height one prime ideal of $L$ which is disjoint from $D$, then $PD^{-1}$ is a height one prime ideal of $\widehat{L}$, so it is generated by a normal element  $b\in\widehat{L}$. It is easy to see that in a quantum coordinate ring the normal elements are $q$-central, so there is $\eta\in\K^{*}$ such that $db=\eta bd$. Thus, by Lemma~\ref{L:rr1:pi}, $P$ is a principal ideal, generated by some normal element $x\in L$.

The statement regarding $u$ follows similarly.
\end{proof}

So it remains to consider the prime ideals that contain both a power of $d$ and a power of $u$. We start by discussing the simpler case where $s$ is not a root of unity.

\begin{proposition}\label{P:rr1:snr1}
Let $L=L(f, r, s, 0)$, with $f\neq 0$  conformal and $r\neq 1$  a  root of unity. If  $s$ is not a root of unity,  then $L$ is a Noetherian UFR, but not a Noetherian UFD. 
\end{proposition}

\begin{proof}
In view of Corollary~\ref{C:rr1:notUFD} and Proposition~\ref{P:rr1:udnP}, it is enough to show that the height one prime ideals of $L$ either do not contain any power of $d$ or do not contain any power of $u$. 

Let $P$ be a prime ideal of $L$ which contains a power of $d$ and a power of $u$. Since $r$ is a root of unity and $s$ is not, it follows that $(s/r^{m})^{k}\neq1$, for all $k>0$. Thus, by Lemma~\ref{L:ql:pk}, the polynomials $P_{k}$ are all nonzero, for $k>0$. Hence, $P=Q_{\lambda}$, for some $\lambda\in\K$, by Theorem~\ref{T:ql}.

If $\lambda=0$, then $h\in P$; otherwise $h^{l}-\lambda^{l}\in P$, where $l\geq 2$ is the order of $r$. Therefore, either $\langle h \rangle\subsetneq P$ or $\langle h^{l}-\lambda^{l} \rangle\subsetneq P$, as $P=Q_{\lambda}$ is not principal, so $P$ has height at least two, by Lemma~\ref{L:ql:h1p}, thus proving our claim.
\end{proof}

In the next lemma we deal with the case in which $s$ is a root of unity. Note that if $r$ and $s$ are roots of unity and $f$ is conformal, then Lemma~\ref{L:ql:pk} guarantees the existence of a positive integer $k$ such that $P_{k}=0$. For any such $k$, the elements $u^{k}$ and $d^{k}$ are normal.

\begin{lemma}\label{L:rsr1}
Let $L=L(f, r, s, 0)$, with $f\neq 0$ conformal, and assume $r$ and $s$ are  roots of unity. Take $k>0$ minimal such that $P_{k}=0$. Then, $u^{k}$ and $d^{k}$ are normal and each generates a height one prime ideal of $L$.  
\end{lemma}

\begin{proof}
We will prove the statement for $u^{k}$; the result for $d^{k}$ will thus follow, by symmetry.

Consider the Ore set $D=\left\{ d^{i}\right\}_{i\geq 0}$ in $L$ and the localisation $\widehat{L}=LD^{-1}$. Recall that $z:=ud-g(h)$ is normal and satisfies $zh=hz$, $dz=szd$ and $zu=suz$ (see Section~\ref{SS:conf}). It is easy to see that $h$ and $z$ generate a (commutative) polynomial algebra in two variables, $\K[h, z]$, and $\widehat{L}=\K[h, z][d^{\pm 1}; \tau]$, where $\tau (h)=rh$, $\tau (z)=sz$, with $u=(z+g(h))d^{-1}$.

Let $\xi=z+g(h)\in\K[h, z]$. This is an irreducible polynomial in the polynomial algebra $\K[h, z]$, hence it generates a prime ideal $P=\xi\K[h, z]$. Furthermore, $\tau^{i}(\xi)$ and $\tau^{j}(\xi)$ are associated irreducible polynomials if and only if $k$ divides $i-j$. The latter follows from the minimality of $k$, as $P_{i}(h)=0\iff s^{i}g(h)=g(r^{i}h) \iff$ $k$ divides $i$.

Thus,
$$
I:=\bigcap_{i\in\mathbb{Z}}\tau^{i}(P)=\bigcap_{i\in\mathbb{Z}}\tau^{i}(\xi)\K[h, z]=\bigcap_{1-k\leq i\leq 0}\tau^{i}(\xi)\K[h, z]
=\prod_{1-k\leq i\leq 0}\tau^{i}(\xi)\K[h, z]
$$
is a $\tau$-prime ideal of $\K[h, z]$. It follows (e.g. by~\cite[Prop.\ 2.1]{aB85}) that $Q:=I\widehat{L}$ is a prime ideal of $\widehat{L}$.

\medskip

\textit{Claim:} $\displaystyle{\prod_{1-n\leq i\leq 0}\tau^{i}(\xi)=u^{n}d^{n}}$, for all $n\geq 0$.

\medskip

The claim above can be  readily established by induction. In particular, $Q=u^{k}d^{k}\widehat{L}=u^{k}\widehat{L}$.

It remains to show that the prime ideal that $Q$ contracts to in $L$ is generated by $u^{k}$. This follows by applying Lemma~\ref{L:rr1:pi} to the contraction of $Q$ to $L$, and noting that:
\begin{itemize}
\item $du^{k}=s^{k}u^{k}d$, and 
\item for $n\in\mathbb{Z}$, $u^{k}d^{n}\in L\iff n\geq 0.$
\end{itemize}

Finally, the height of $\langle u^{k} \rangle$ is one, by the Principal Ideal Theorem.
\end{proof}

Our final result finishes the classification of which generalized down-up algebras are Noetherian UFR's.

\begin{theorem}\label{T:rr1:UFR}
Let $L=L(f, r, s, 0)$, with $f\neq 0$ conformal and $r\neq 1$ a root of unity. Then $L$ is a Noetherian UFR but not a Noetherian UFD. 
\end{theorem}

\begin{proof}
By Proposition~\ref{P:rr1:snr1}, it remains to consider the case where $s$ is a root of unity (possibly equal to $1$), and by Corollary~\ref{C:rr1:notUFD} and Proposition~\ref{P:rr1:udnP}, it will be enough to show that there are no height one prime ideals of $L$ which contain both a power of $d$ and a power of $u$. 

Let $P$ be a prime ideal of $L$ which contains a power of $d$ and a power of $u$. Let $k>0$ be minimal such that $P_{k}=0$. Since $u^{k}$ is normal, we must have $u^{k}\in P$, so P contains the height one prime ideal $\langle u^{k} \rangle$, by Lemma~\ref{L:rsr1}. So $P$ does not have height one, as  $\langle u^{k} \rangle$ contains no power of $d$.
\end{proof}

%%%%%%%%%%%%%%%%%%%%%%%%%%%%%%%%%%%%%%%%%%%%%%%%%%%%%%%%%%%%%%%%%%
%%%%%%%%     Section 7                              %%%%%%%%%%%%%%
%%%%%%%%%%%%%%%%%%%%%%%%%%%%%%%%%%%%%%%%%%%%%%%%%%%%%%%%%%%%%%%%%%

\section{Proofs of Theorems A and B}

In this final section, we start by proving Theorem~B, which gives  a complete classification of the generalized down-up algebras which are a Noetherian UFR, and then we prove Theorem~A.  We also specialise our results to down-up algebras, as introduced by Benkart and Roby in~\cite{BR98}.

%\begin{theorem}\label{mainUFR}
%Let $L=L(f, r, s, \gamma)$ be a generalized down-up algebra with $rs \neq 0$. Then $L$ is a Noetherian UFR  except if $f\neq 0$ and one of the following conditions is satisfied:
%\begin{enumerate}
%\item $f$ is not conformal, $r$ is not a root of unity, and there exists $\zeta\neq \gamma/(r-1)$ such that $f(\zeta)=0$;
%\item $f$ is conformal, $\langle r, s\rangle$ is a free abelian group of rank $2$, and there exists $\zeta\neq \gamma/(r-1)$ such that $f(\zeta)=0$;
%\item $\gamma\neq 0$, $r=1$, $s$ is not a root of unity, and $f\notin \K$.
%\end{enumerate}
%\end{theorem}

\begin{proof}[Proof of Theorem B]
Assume first that $\gamma=0$. Then the condition \emph{there exists  $\zeta\neq \gamma/(r-1)$ such that $f(\zeta)=0$} is equivalent to the condition \emph{$f$ is not a monomial}, and the condition \emph{$\langle r, s\rangle$ is a free abelian group of rank $2$} is equivalent to the condition \emph{$r$ is not a root of unity and $\tau=0$}. Thus, in this case, the result follows from Theorem~\ref{T:fnconf:class}, Proposition~\ref{P:f0:g0:UFR}, Theorem~\ref{T:rnru:fc:UFR}, Theorem~\ref{T:r1:g0:snru:UFD}, Theorem~\ref{T:r1:g0:sru:UFR} and Theorem~\ref{T:rr1:UFR}.

Now assume that $\gamma\neq 0$ and $r=1$. Then, by Proposition~\ref{P:conf:ris1}, $f$ is conformal, and the result follows from Proposition~\ref{P:f0:gn0:UFR} and Theorem~\ref{T:r1:gno:UFR}.

Finally, if $\gamma\neq 0$ and $r\neq 1$, then Proposition~\ref{P:conf:gamma0rnot1} asserts that $L$ is isomorphic to a generalized down-up algebra $L(\tilde{f}, r, s, 0)$, such that $f$ is conformal in $L(f, r, s, \gamma)$ if and only if $\tilde{f}$  is conformal in $L(\tilde{f}, r, s, 0)$. Furthermore, by the proof of this result (see~\cite[Prop.\ 1.7]{CL09}), we can take $\tilde{f}(h)=f(\frac{h+\gamma}{r-1})$. Hence, in this case, the result follows from applying our previously established criteria to $L(\tilde{f}, r, s, 0)$.
\end{proof}

To finish the classification, we just need to determine the generalized down-up algebras which are a Noetherian UFD, and prove Theorem A.

%\begin{theorem}\label{mainUFD}
%Let $L=L(f, r, s, \gamma)$ be a generalized down-up algebra with $rs \neq 0$. Then $L$ is a Noetherian UFD if and only if $L$ is a Noetherian UFR and $\langle r, s \rangle$ is torsionfree.
%\end{theorem}

\begin{proof}[Proof of Theorem A]
It will be enough to establish this result in the case $\gamma=0$, and the case $\gamma\neq 0$, $r=1$, by Proposition~\ref{P:conf:gamma0rnot1}, as the statement does not involve $f$ or $\gamma$. So we assume that either $\gamma=0$ or $r=1$. 
\begin{itemize}
\item If $f$ is not conformal then $\gamma=0$, by Proposition~\ref{P:conf:ris1}, and thus, by Lemma~\ref{L:conf:gamma0}, $\langle r, s \rangle=\langle r \rangle$. Then Theorem~\ref{T:fnconf:class} establishes the result.
\item If $f=0$ and $\gamma=0$, then the result follows from Propositions~\ref{P:f0:g0:UFR} and~\ref{P:f0:g0:UFD}.
\item If $f=0$ and $\gamma\neq 0$, then we assume $r=1$ and the result follows from Proposition~\ref{P:f0:gn0:UFR} and Theorem~\ref{T:f0:gn0:UFD}.
\item If $f\neq0$ is conformal and $r$ is not a root of unity, then we assume $\gamma=0$  and the result follows from Theorems~\ref{T:rnru:fc:UFR} and~\ref{T:rnru:fc:UFD}.
\item If $f\neq0$ is conformal, $r=1$ and $\gamma\neq 0$, then Theorems~\ref{T:r1:gno:UFR} and~\ref{T:r1:gno:UFD} establish the result.
\item If $f\neq0$ is conformal, $r=1$ and $\gamma=0$, then Proposition~\ref{P:conf:ris1} implies that $s\neq 1$. Thus, Theorems~\ref{T:r1:g0:snru:UFD} and~\ref{T:r1:g0:sru:UFR} imply the result.
\item If $r\neq 1$ is a root of unity, then we can assume that $\gamma=0$, and the result follows directly from Corollary~\ref{C:rr1:notUFD}.
\end{itemize}
\end{proof}

We note that the hypothesis that $L$ be a Noetherian UFR, in Theorem~A, is essential, as the following example illustrates. Let $r\in\K^{*}$ be a non-root of unity, $s\in\{1, r\}$ and $f=h\in\K[h]$.  Then $L=L(h, r, s, 1)$ is not a Noetherian UFD, by Proposition~\ref{P:conf:gamma0rnot1} and Theorem~B(a). Yet, $\langle r, s\rangle\simeq \mathbb{Z}$ is torsionfree. Notice that $L(h, r, s, 1)$ is isomorphic to the down-up algebra $A(r+s, -rs, 1)$.

In general, the down-up algebra $A(\alpha, \beta, \gamma)$, as defined in~\cite{BR98}, can be viewed as the generalized down-up algebra $L(h, r, s, \gamma)$, where $\alpha=r+s$ and $\beta=-rs$ (see~\cite[Lem.\ 1.1]{CL09} for more details). So we have:

\begin{corollary}\label{mainDUA}
Let $A=A(\alpha, \beta, \gamma)$ be a down-up algebra over $\K$ with $\beta\neq 0$. Let $r, s\in\K$ be the roots of $h^{2}-\alpha h-\beta$. Then $A$ is a Noetherian UFR  except if $\gamma\neq 0$, $\beta$ is not a root of unity and one of the following conditions is satisfied:
\begin{enumerate}
\item $\alpha+\beta=1$;
\item $\alpha^{2}+4\beta=0$;
\item $\langle r, s \rangle$ is a free abelian group of rank $2$.
\end{enumerate}
Furthermore, $A$ is a Noetherian UFD if and only if $A$ is a Noetherian UFR and $\langle r, s \rangle$ is torsionfree.
\end{corollary}

\begin{proof}
We use the isomorphism $A(\alpha, \beta, \gamma)\simeq L(h, r, s, \gamma)$. First, by Proposition~\ref{P:conf:gamma0rnot1}, Lemma~\ref{L:conf:gamma0} and Proposition~\ref{P:conf:ris1}, we conclude that $f(h)=h$ is conformal in $L(h, r, s, \gamma)$ if and only if one of the following conditions holds:
\begin{itemize}
\item $\gamma=0$ and $r\neq s$;
\item $\gamma\neq 0$, $r\neq 1$, $s\neq 1$ and $r\neq s$;
\item $\gamma\neq 0$ and $r=1$.
\end{itemize}
Thus, we can apply Theorem~B to conclude that $A$ is a Noetherian UFR except in the cases listed below:
\begin{itemize}
\item $\gamma\neq 0$, $s=1$ and $r$ is not a root of unity;
\item $\gamma\neq 0$, $s=r$ and $r$ is not a root of unity;
\item $\gamma\neq 0$ and $\langle r, s \rangle$ is a free abelian group of rank $2$;
\item $\gamma\neq 0$, $r=1$ and $s$ is not a root of unity.
\end{itemize}
Notice that, in all of these cases, $\gamma\neq 0$ and $\beta=-rs$ is not a root of unity. Also, $\alpha+\beta=1\iff r=1$ or $s=1$, and $\alpha^{2}+4\beta=0\iff r=s$.  The first part of the theorem thus follows. The second part is a direct consequence of Theorem~A.
\end{proof}

Two down-up algebras of particular interest are the enveloping algebra of the Lie algebra $\mathfrak{sl}_{2}$ and the enveloping algebra of the $3$-dimensional Heisenberg Lie algebra, which occur as $A(2, -1, 1)$ and $A(2, -1, 0)$, respectively. Using Corollary~\ref{mainDUA}, we retrieve the well-known fact that each of these two algebras is a Noetherian UFD (see~\cite{nC74} and \cite[Prop.\ 3.1]{C84}).  

Generalized down-up algebras also include other classes of algebras, such as Smith's algebras \cite{spS90} and Rueda's algebras \cite{sR02}. In the case of Smith's algebras, the result is quite straightforward.
Let $f \in \mathbb{K}[H]$. Recall that the Smith algebra $S(f)$ is the $\mathbb{K}$-algebra generated by $A,B,H$ with relations:
$$[H,A]=A,~[H,B]=-B \mbox{ and } [A,B]=f(H).$$
It is well known that $S(f) \simeq L(f,1,1,1)$. Hence, we deduce from Theorems A and B the following result.

\begin{corollary}
Let $S(f)$ be a Smith algebra with  $f \in \mathbb{K}[H]$. Then, $S(f)$ is a Noetherian UFD.
\end{corollary}

%%%%%%%%
%%%%%%%%%%%%%%%%%%%%%%%%%%%%%%%%%%%%%%%%%%%%%%%%%%%%%%%%%%%%%%%%%%
%%%%%%%%     Bibliography                           %%%%%%%%%%%%%%
%%%%%%%%%%%%%%%%%%%%%%%%%%%%%%%%%%%%%%%%%%%%%%%%%%%%%%%%%%%%%%%%%%

\noindent
St\'ephane Launois\\ 
School of Mathematics, Statistics \& Actuarial Science,\\ 
University of Kent, Canterbury, 
Kent CT2 7NF, United Kingdom\\
E-mail: S.Launois@kent.ac.uk
 \\[10pt]
Samuel A.\ Lopes\\
Centro de Matem\'atica da Universidade do Porto,\\
Universidade do Porto,\\
Rua do Campo Alegre 687, 4169-007 Porto, Portugal\\
E-mail: slopes@fc.up.pt

\end{document}